\documentclass[11pt]{amsart}
\usepackage{amsmath}
\usepackage{amsthm}
\usepackage{bbm}
\usepackage{amsfonts}
\usepackage{euscript}
\usepackage{amssymb}
\usepackage{enumerate}
\usepackage{mathrsfs}
\usepackage{mathtools}
\usepackage{dsfont}
\usepackage[english]{babel}
\usepackage[backend=biber,style=alphabetic]{biblatex}
\usepackage{geometry}
\usepackage{comment}
\usepackage{bbm}
\usepackage[colorlinks,linkcolor = blue,citecolor = blue]{hyperref}
\usepackage{cleveref}
\usepackage{float}
\usepackage{xfrac}
\usepackage{csquotes}
\usepackage{microtype}
\usepackage[normalem]{ulem}

\newcommand{\field}[1]{\mathbb{#1}}
\newcommand{\R}{\field{R}}

\newcommand{\N}{\field{N}}

\renewcommand{\L}{\field{L}}
\newcommand{\T}{\mathbb{T}}

\newcommand{\wto}{\rightharpoonup}
\newcommand{\wsto}{\overset{\ast}{\rightharpoonup}}
\let\div\undefined
\DeclareMathOperator{\div}{div}
\newcommand{\beq}{\begin{equation}}
\newcommand{\eeq}{\end{equation}}

\renewcommand{\phi}{\varphi}
\renewcommand{\theta}{\vartheta}

\DeclareMathOperator*{\esssup}{ess\,sup}
\DeclareMathOperator*{\essinf}{ess\,inf}
\let\S\relax
\usepackage{cancel}

\newcommand{\S}{\mathbb{S}}

\newtheorem{theorem}{Theorem}

\newtheorem{lemma}{Lemma}
\newtheorem{definition}{Definition}
\newtheorem{proposition}{Proposition}

\theoremstyle{remark}
\newtheorem{remark}{Remark}[section]

\numberwithin{equation}{section}
\numberwithin{theorem}{section}
\numberwithin{lemma}{section}
\numberwithin{definition}{section}

\usepackage{import}
\usepackage{xifthen}
\usepackage{pdfpages}
\usepackage{transparent}

\newcommand{\red}[1]{\textcolor{red}{#1}}
\title[stability of entropy waves]{Uniqueness and stability of entropy waves among vanishing dissipation limits from the Navier--Stokes--Fourier System}
\date{\today}

\author[Cooper Faile]{Cooper Faile}
\address[Cooper Faile]{\newline Department of Mathematics \newline The University of Texas at Austin, \newline 2515
Speedway, Austin, TX 78712 \newline USA}
\email{jcfaile@utexas.edu}

\author[Ingrid Lacroix-Violet]{Ingrid Lacroix-Violet}
\address[Ingrid Lacroix-Violet]{\newline Universit\'e de Lorraine, CNRS, IECL, \newline F-54000 Nancy, \newline France}
\email{ingrid.violet@univ-lorraine.fr}

\author[Alexis F. Vasseur]{Alexis F. Vasseur}
\address[Alexis F. Vasseur]{\newline Department of Mathematics \newline The University of Texas at Austin, \newline 2515
Speedway, Austin, TX 78712 \newline USA}
\email{vasseur@math.utexas.edu}

\newcommand{\norm}[1]{\left\lVert#1\right\rVert}
\thanks{2020 \textit{Mathematics Subject Classification}. Primary: 35Q31, Secondary: 35Q30, 35L67, 35B35}
\thanks{\textit{Key words and phrases}. Compressible Euler system, Navier Stokes Fourier system, Uniqueness, Stability, Vanishing dissipation limit, Entropy wave}
\thanks{\textbf{Acknowledgment.} This work was partially supported through the National Science Foundation grants DMS-2306852, DMS-1840314 and a joint NSF-EPSRC grant DMS-EPSRC-2219434.
The second and third authors are grateful to the IHES for their hospitality and support during their visit in Spring 2026.}

\begin{document}

\definecolor{vert}{rgb}{0.3,0.7,0.4}
\begin{abstract}
In this paper, we establish the stability of general entropy wave solutions of the compressible Euler equations (that is, profiles with constant velocity and pressure) in the class of vanishing dissipation limits from the Navier--Stokes--Fourier system. 
If the initial data of our sequence of solutions to the Navier--Stokes--Fourier system converges to the initial data of the entropy wave solution, we further have convergence of the solutions to the entropy wave. 
This gives uniqueness and stability of the entropy wave in this class. 
The proof is based on the $a$-contraction theory.
\end{abstract}

\maketitle

\section{Introduction}
Consider the compressible Euler system,
\begin{equation} \label{eq:euler-system}
    \begin{cases}
        \begin{aligned}
            \rho_t + \div (\rho u) &= 0, \\
            (\rho u)_t + \div (\rho u\otimes u) + \nabla p_{M}(\rho,\theta) &= 0, \\
             \left(\frac{1}{2}\rho |u|^2 + \rho e_M\right)_t + \div \left( \left(\frac{1}{2}\rho |u|^2 + \rho e_M\right) u + p_M(\rho,\theta) u \right) &= 0,
        \end{aligned}
    \end{cases}
\end{equation}
for density $\rho$, velocity $u$, temperature $\theta$, specific internal energy $e_M$, and pressure $p_M$ of a compressible gas.
This system possesses a special family of solutions called \emph{entropy waves}.
These solutions are characterized by the pressure $p_M$ and velocity $u$ being constant within the fluid.
The density $\rho$ and entropy $\rho s_M$ of the fluid may be non-constant and are simply transported by the constant velocity $u$. 
Generically solutions to~\eqref{eq:euler-system} form singularities in finite time, and it is possible for an entropy wave to possess discontinuities (for instance, the planar contact discontinuity without shear).
As a result, we must consider weak solutions where the system~\eqref{eq:euler-system} holds in some weaker sense.

This class of distributional solutions is too broad to provide uniqueness or stability of these solutions. 
One way to address this issue is by selecting admissible solutions which satisfy an entropy inequality,
\begin{equation} \label{eq:entropy-ineq}
    (\rho s_M)_t + \div (\rho s_M u) \geq 0.
\end{equation}
Indeed, it is known by the work of Dafermos \cite{MR546634} and DiPerna \cite{MR523630} that any bounded Lipschitz solution to the Euler system~\eqref{eq:euler-system} is unique within the class of bounded weak solutions satisfying the entropy inequality~\eqref{eq:entropy-ineq}.
However, the stability of discontinuous solutions is much more delicate. 
In dimensions two and greater the entropy inequality~\eqref{eq:entropy-ineq} is known to be an insufficient criterion to select a unique solution for certain initial data (see \cite{de2010admissibility}, \cite{giri2022non}, \cite{chiodaroli2015global}, and \cite{markfelder2021convex} for such constructions via convex integration).
Furthermore, there is currently no global existence theory for distributional weak solutions satisfying the entropy inequality in the multi-d setting.
Another more stringent admissibility criterion is to select solutions that may be obtained as a limit of solutions to a more physical system. 
In this case, we consider vanishing dissipation limits from the following Navier--Stokes--Fourier system,
\begin{equation} \label{eq:nsf-system}
    \begin{cases}
        \begin{aligned}
            \rho_t + \div (\rho u) &= 0, \\
            (\rho u)_t + \div (\rho u\otimes u) + \nabla p(\rho,\theta)  &= \nu\,\div(\S(\theta, \nabla u)), \\
            \left(\frac{1}{2}\rho |u|^2 + \rho e\right)_t + \div \left(\left(\frac{1}{2}\rho |u|^2 + \rho e\right) u - \nu\,\S u+pu+\kappa\mathbf{q}\right)   &= 0.
        \end{aligned}
    \end{cases}
\end{equation} 
Following Feireisl and Novotn\'{y} \cite{feireislnovotny17}, the pressure in the above system takes the form $p = p_M + ap_R$, where $p_R$ is a radiative pressure depending solely on the temperature $\theta$ and the parameter $a > 0$ modulates the strength of this pressure. 
The internal energy takes a similar form, $e = e_M + ae_R$.
We additionally have the viscous stress tensor $\S(\theta, \nabla u)$ and heat flux $\mathbf{q}(\theta,\nabla \theta)$, which introduce a dissipation in $u$ and $\theta$ that is not present in the Euler system~\eqref{eq:euler-system}.
Additionally, the strengths of the dissipative terms $\S$ and $\mathbf{q}$ are modulated by the parameters $\nu$ and $\kappa > 0$, respectively.
A complete discussion of these constitutive relations for the system~\eqref{eq:nsf-system} is in \Cref{sec:con-rel}. 
We now define the vanishing dissipation solutions; while they may not be distributional solutions to the Euler system~\eqref{eq:euler-system} they are a broad class of objects which represent the dynamics of a compressible inviscid fluid.
\begin{definition} \label{def:vanishing-diss-soln}
    A \emph{vanishing dissipation solution} to the system~\eqref{eq:euler-system} with initial data $U_0 = (\rho_0,u_0, \theta_0)$ is a triple $U = (\rho, m, \mathcal{E})$ obtained by respective weak limits from solutions to the Navier--Stokes--Fourier system~\eqref{eq:nsf-system} (see \Cref{def:solution}) in the following sense: for all $\kappa, a,\nu > 0$ let $U^{\kappa,a,\nu} = (\rho^{\kappa,a,\nu},m^{\kappa,a,\nu},\mathcal{E}^{\kappa,a,\nu}) = (\rho^{\kappa,a,\nu}, \rho^{\kappa,a,\nu}u^{\kappa,a,\nu}, \frac{|m^{\kappa,a,\nu} |^2}{2\rho^{\kappa,a,\nu}} + \rho^{\kappa,a,\nu}(e_M(\rho^{\kappa,a,\nu},\theta^{\kappa,a,\nu}) + ae_R(\rho^{\kappa,a,\nu},\theta^{\kappa,a,\nu}) ) $ be an admissible solution to the Navier--Stokes--Fourier system with these respective parameters taking the fixed initial data $U^0$. 
    Then, sequentially taking the limits $\kappa \to 0$, $a \to 0$, and $\nu \to 0$ along a subsequence we attain the limit $U$ for all $T > 0$ in the following weak topologies:
    \begin{align*}
        \rho^{\kappa,a,\nu} &\wsto \rho \quad \text{weakly$-\ast$ in}\quad L^\infty(0,T; L^2(\T^3)), \\
        m^{\kappa,a,\nu} &\wsto m \quad \text{weakly$-\ast$ in}\quad L^\infty(0,T; L^{4/3}(\T^3)), \\
        \mathcal{E}^{\kappa,a,\nu} &\wsto \mathcal{E}\quad\text{weakly$-\ast$ in}\quad L^\infty(0,T; \mathcal{M}(\T^3)),  
    \end{align*}
    where $\mathcal{M}$ denotes the space of Radon measures on $\T^3$.
\end{definition}

There have been several works on the uniqueness of entropy waves and related solutions to the Euler equations~\eqref{eq:euler-system}. 
The case of planar contact discontinuities without shear among vanishing dissipation limits from a Navier--Stokes--Fourier system via the method of $a$-contraction was first established by the third author and Wang in the 1-d case \cite{vy-1d} and later Kang, the third author, and Wang proved in the 3-d case \cite{kvy21}. 
We note that the existence of weak solutions to the Navier--Stokes--Fourier system for the constitutive relations considered in \cite{kvy21} remains open.
Wang, Zhang, and Zhao proved the nonlinear stability of entropy waves under the Taylor sign condition.
Specifically, they describe the evolution of the boundary of an interface of discontinuity, and show stability provided the Taylor sign condition holds \cite{wang2024nonlinear}.
In this work, we study waves with $p_M \equiv \text{const}$, so we are outside the scope of \cite{wang2024nonlinear}.
We note that due to the constant velocity assumption, our result does not apply to vortex sheets or contact discontinuities with shear. 
These structures are well known to exhibit Kelvin--Helmholtz instability and can admit infinitely many weak solutions.
For such examples, see \cite{krupa2025contact} where the authors study 2-d compressible Euler with a particular pressure and \cite{mengual2023dissipative} where the authors show non-uniqueness of vortex sheets for 2-d incompressible Euler.
In the setting of smooth solutions, Feireisl proved stability under vanishing dissipation limits from the Navier--Stokes--Fourier system, with a specific relation between the radiative coefficient $a$ and dissipation coefficients $\nu$ and $\kappa$ \cite{feireisl-zero-diss-limit}.


Compared with previous work \cite{kvy21}, we focus on establishing stability in the regime of constitutive relations of the system~\eqref{eq:nsf-system} for which there is an established existence theory. 
For this purpose, the theory developed by Feireisl and Novotn\'{y} \cite{feireislnovotny17} provides the existence of weak solutions of~\eqref{eq:nsf-system} for a large family of perturbations of our entropy wave. 
Additionally, we establish stability of a much more general family of profiles. 

\subsection{Main results}
We begin by stating an existence theorem for the vanishing dissipation solutions \Cref{def:vanishing-diss-soln} which holds for a general class of initial data. 
First, we must mention that our molecular pressure is assumed to satisfy
$$ p_M(\rho,\theta) \sim \rho^\gamma \quad \text{for}\quad \theta \ll \rho $$
for some $\gamma > 2$.
The specific form of this pressure will be made precise in \Cref{sec:con-rel}.
We now state the existence theorem. 
\begin{theorem}
    \label{thm:existence}
    If $U_0 = (\rho_0, u_0, \theta_0)$ satisfies that 
    \begin{gather*}
        \rho_0 \geq 0,\quad  \int_{T^3} \rho_0\,dx > 0,\\  \int_{\T^3} \rho_0 \frac{|u_0|^2}{2}+ \rho_0 e_M(\rho_0,\theta_0) + \rho_0|s_M(\rho_0,\theta_0)| + \theta_0^4\,dx < \infty,\\\theta_0>0\ \text{a.e. in }\T^3
    \end{gather*}
    then there exists a vanishing dissipation solution $U$ to the system~\eqref{eq:euler-system} with initial data $U_0$. 
\end{theorem}

We now discuss the stability of the entropy wave. 
Without loss of generality, we may assume by Galilean invariance that any entropy wave is stationary. 
In this frame, we fix an entropy wave with density $\bar \rho_0 $ and a temperature $\bar \theta_0$ which we suppose are both bounded above and below away from zero.
This wave can be expressed in conservative variables as
\begin{equation}
    \label{eq:V-form}
    V(t,x) = ( \bar\rho_0(x), 0, \bar{\mathcal{E}} ),
\end{equation} 
where $\bar{\mathcal{E}} = \bar \rho_0(x) e_M(\bar \rho_0(x), \bar \theta_0(x))$. 
We now state the stability theorem for this wave. 
\begin{theorem}
    \label{thm:main}
    Let $U_0^n=(\rho_0^n,u_0^n,\theta_0^n)$ be a sequence of initial data satisfying
    \begin{equation}
        \label{eq:uniform-bound-assumption-thm}
        \begin{gathered}
            \sup_{n\in \N} \left(\norm{\rho_0^n}_{L^\gamma(\T^3)} + \norm{\rho_0^n\frac{|u_0^n|^2}{2} + \rho_0^ne_M(\rho_0^n,\theta_0^n)}_{L^1(\T^3)} + \norm{\theta_0^n}_{L^4(\T^3)} \right)< \infty \\ 
            \left(\rho_0^n,\rho_0^n u_0^n, \rho_0^n \frac{|u_0^n|^2}{2} + \rho_0^n e_M(\rho_0^n,\theta_0^n) \right) \to \left(\bar\rho_0, 0,  \bar{\mathcal{E}}\right) \quad \text{in} \quad L^2(\T^3),
        \end{gathered}
    \end{equation}
    where $\bar \rho_0, \bar{\mathcal{E}}$ are conservative coordinates of the entropy wave~\eqref{eq:V-form}.
    If $U_n = (\rho_n, m_n, \mathcal{E}_n)$ is a sequence of vanishing dissipation solutions to the system~\eqref{eq:euler-system} with the corresponding initial data $U_0^n$ in the sense of \Cref{def:vanishing-diss-soln} then we have 
    \begin{equation} \label{eq:main-thm-cvs}
        \begin{aligned}
            \rho^n \wsto&\, \bar\rho \text{ weakly-$\ast$ in }L^\infty(0,T; L^2(\T^3)), \\
            m^n \to&\, 0  \text{ in }L^\infty(0,T; L^{4/3}(\T^3;\R^3)), \\
            \mathcal{E}^n\to&\, \bar{\mathcal{E}}  \text{ in }L^\infty(0,T;\mathcal{M}(\T^3)).
        \end{aligned}
    \end{equation}
\end{theorem}

\subsection{Remarks on \Cref{thm:main}} \label{sec:rmkthm1}
\begin{itemize}
    \item Contrary to \cite{kvy21}, we take constitutive relations which allow us to establish the existence of vanishing dissipation limits for the Euler system. 
        In \cite{kvy21} the pressure was taken to be of the form $p(\rho,\theta) = R\rho \theta + C\rho^\gamma$ for some fixed constants $R,C > 0$ and $\gamma > 2$. 
        Notably, this pressure does not provide an a priori estimate on the integrability of the temperature $\theta$ which leads to difficulties in constructing solutions (in the sense of admissible weak solutions \Cref{def:solution}) due to the presence of temperature dependent viscosity.
        The existence of solutions to the Navier--Stokes--Fourier system without such a radiative pressure term is an outstanding open problem \cite{chiodaroli2025long,novotnyexist13}. 
        In this paper our pressure satisfies the assumptions that are included in the Feireisl existence theory \cite{feireislnovotny17,novotnyexist13}.

    \item When the underlying entropy wave obtains just two values (as in the case of a planar contact discontinuity) the weak convergence of density in~\eqref{eq:main-thm-cvs} can be improved to a strong convergence in $L^q(0,T; L^p(\T^3))$ for all $q < \infty$ and $p < \gamma$. See \Cref{rmk:strong-cv-density} for further details.
        
    \item In this paper we suppose that our molecular pressure satisfies $p_M(\rho,\theta) \sim \rho^\gamma$ when $ \theta  \ll\rho $ for some constant $\gamma > 2$. Note that it is known that solutions to the Navier--Stokes--Fourier system~\eqref{eq:nsf-system} exist for any $\gamma > \frac{3}{2}$ \cite{feireislnovotny17}. Here the hypothesis $\gamma>2$ is due to a technical step in the vanishing dissipation limit requiring an estimate of $\rho$ and $\rho s$ in $L^\infty(0,T,L^2(\T^3))$ that is uniform in $\kappa >0$. Our argument for this estimate requires that $\gamma>2$ (see \Cref{lemma:rho-s-integrability} for more details). Of course, it would be interesting to investigate lower values of $\gamma$ and, in particular, the physical case of a monoatomic gas $\gamma=5/3$ which is included in the existence theory. 
    But at the moment that does not seem possible with the method developed here. 

\end{itemize}

\subsection{Proof strategy}
The proof of \Cref{thm:existence} follows quickly after proving a priori estimates on admissible Navier--Stokes--Fourier solutions.

To prove \Cref{thm:main} we use a strategy based on the relative entropy method of Dafermos and DiPerna \cite{MR546634,MR523630}.
Given a convex entropy $\eta(U)$ (for the Euler system $\eta(U) = -\rho s_M$) and a fixed state $V$ we define the relative entropy 
\begin{equation}
\label{eq:def_relative_entropy}
\eta(U|V) := \eta(U) - \eta(V) - \eta'(V)(U-V)
\end{equation}
where the derivative $\eta'$ is taken in the conservative variables (for instance, $(\rho, m, \mathcal{E})$ in the case of the system~\eqref{eq:euler-system}).
Before proving \Cref{thm:main} we must first define two convex functions,
\begin{align}
    \mathcal{G}(x)  & := \begin{cases} 
        | x |^2     & | x | < 1 \\
        2| x | -1   & | x| \geq 1,
    \end{cases} \label{eq:G-coercive} \\
    \mathcal{F}(\rho) &:= \begin{cases}
        \min(\mathcal{G}(\rho - \rho_*),\, \mathcal{G}(\rho - \rho^*)) & \rho \notin [\rho_*, \rho^*] \\
        0 & \text{otherwise,}
    \end{cases} \label{eq:F-coercive}
\end{align}
where 
\begin{equation} \label{eq:density-bounds}
    \rho_* := \essinf_{x \in \T^3}(\bar \rho_0(x))\quad\text{and}\quad \rho^* := \esssup_{x \in \T^3}(\bar \rho _0(x) )
\end{equation}are the extreme densities achieved by the entropy wave $V$. 
To prove \Cref{thm:main} we will actually prove the following more general proposition. 
\begin{proposition} \label{prop:main}
    If an initial data $U_0$ satisfies
    \begin{equation}\label{eq:prop-integrability-for-stability}
        \left(\norm{\rho_0}_{L^\gamma(\T^3)} + \norm{\rho_0\frac{|u_0|^2}{2} + \rho_0e_M(\rho_0,\theta_0)}_{L^1(\T^3)} + \norm{\theta_0}_{L^4(\T^3)} \right)< \infty
    \end{equation}
    then any vanishing dissipation solution $U = (\rho, m, \mathcal{E})$ corresponding to this initial data satisfies
    \begin{multline} \label{eq:main-est} \esssup_{t \in [0,T]} \int_{\T^3} \mathcal{G}(m)+ \mathcal{F}(\rho)\,dx +\norm{\mathcal{E} - \bar{\mathcal{E}}}_{L^\infty(0,T; \mathcal{M}(\T^3))} \\\leq C \sqrt{\int_{\T^3} (-\rho s_M)(U_0| V)\,dx} + C\int_{\T^3} (-\rho s_M)(U_0| V)\,dx.\end{multline}
    Additionally, if we have a sequence of initial data $\{U_0^n\}_{n=1}^\infty$ such that 
    \begin{equation} \label{eq:rel-ent-converge-init-data} 
        \int_{\T^3} (-\rho s_M)(U_0^n| V)\,dx \to 0
    \end{equation}
    then for any sequence of vanishing dissipation limits $\{(\rho^n, m^n, \mathcal{E}^n)\}$ satisfying~\eqref{eq:main-est} we have that
    \begin{equation} \label{eq:main-prop-cvs}
        \begin{aligned}
            \rho^n \wsto&\, \bar\rho \text{ weakly-$\ast$ in }L^\infty(0,T; L^\gamma(\T^3)), \\
            m^n \to&\, 0  \text{ in }L^\infty(0,T; L^{q}(\T^3;\R^3)) \quad \text{for all}\quad q < 2\tfrac{\gamma}{\gamma+1}, \\
            \mathcal{E}^n\to&\, \bar{\mathcal{E}}  \text{ in }L^\infty(0,T;\mathcal{M}(\T^3)),
        \end{aligned}
    \end{equation}
    establishing the uniqueness of the entropy wave of the Euler system~\eqref{eq:euler-system} among vanishing dissipation limits from the Navier--Stokes--Fourier system~\eqref{eq:nsf-system}.
\end{proposition}
The proof of \Cref{prop:main} is based on the method of $a$-contraction with shifts, which was pioneered by the third author and his collaborators in the inviscid case \cite{MR2807139,kang2016criteria}. 
This is a robust and fully non-perturbative stability method for conservation laws to handle singular patterns, able to handle perturbations of both initial data and the system. 
It has successfully been used to obtain uniform in the viscosity stability estimates of viscous shocks \cite{kang2020contraction}, perform the inviscid limit from Navier--Stokes to Euler in 1-d \cite{NSVS,kang2021uniqueness}, take the inviscid limit in 3-d for contact discontinuities \cite{kvy21}, and establish the time asymptotic stability of the Riemann problem for the Navier--Stokes--Fourier system \cite{kang2025time}.

We now discuss the proof.
The relative entropy $\eta(\cdot | V)$ for a fixed state $V$ is another entropy of our system~\eqref{eq:euler-system} and due to the strict convexity of $\eta$ in the conservative variables it (locally) satisfies the quadratic estimate $\eta(U|V) \sim |U-V|^2$, where $a\sim b$ is understood to mean that there exists $C > 0$ such that $C^{-1}b \leq a \leq Cb$. 
In the 1-d case for a single discontinuity, we let 
$$ V(t,x) = \begin{cases}
    V_L & x < h(t) \\
    V_R & x > h(t)
\end{cases} $$
be a discontinuity shifted by $h$. 
It follows that for all bounded functions $U$ and constant weights $a_L, a_R > 0$ we have
\begin{equation} \label{eq:relative-entropy-1d}
    \norm{U(t,\cdot)-V(t,\cdot)}_{L^2}^2 {\sim} \int_{-\infty}^{h(t)} a_L\eta(U|V_L)\,dx + \int_{h(t)}^\infty a_R\eta(U|V_R)\,dx \qquad 
\end{equation}
and, provided that we can produce weights $a_L,a_R$ and a shift $h$ such that the right of~\eqref{eq:relative-entropy-1d} is decreasing for all perturbations $U$ in some class, we can establish that the discontinuity $V_L,V_R$ is stable up to a shift in $L^2$ within this class. 
When $V_{LR}$ are the states of a contact discontinuity of the full Euler system with entropy $ \eta(U) = -(\rho s_M)(U)$ we have this stability when $a_{LR} = \theta_{LR}$ (the temperatures associated to the states $V_{LR}$) and the shift satisfies $\dot h(t) = u(t,h(t))$, the velocity associated to the weak solution $U$ (provided that $u$ is sufficiently regular; see \cite[Section 4]{serre2014relative} and \cite[Corollary 6.2]{serre20142} for further details). 

In analogy with this, one would anticipate that in the case of multi-dimensional entropy waves we could simply solve the transport equation 
\begin{equation} \left\{\begin{aligned}
    \partial_t \bar V + u\cdot \nabla \bar V &= 0 \\ 
    \bar V(0,\cdot) &= V(0,\cdot)
\end{aligned}\right. \end{equation}
where $V$ is the entropy wave~\eqref{eq:V-form} to get a satisfactory notion of shift.
From this we would hope to show that the quantity
\begin{equation} \label{eq:weighted-euler-example} 
\int_{\T^3} \theta(\bar V(t,x)) (-\rho s) (U| \bar V)\,dx \end{equation}
is decreasing in time.
However, this cannot be done at the level of the inviscid Euler system for $d \geq 2$. 
Indeed, as previously mentioned there is no general existence theory for global in time weak solutions to~\eqref{eq:euler-system} and, additionally, we would need that the associated velocity field is at least Sobolev to ensure the above transport equation possesses a solution. 
These issues are alleviated by instead considering weak solutions to the Navier--Stokes--Fourier system~\eqref{eq:nsf-system},  for which  $ u \in L^2(0,T; H^1(\T^3))$.
Formally, solutions to the Navier--Stokes--Fourier system converge to those of the full Euler system~\eqref{eq:euler-system} in the vanishing dissipation limit, so by taking $\kappa,a,\nu\to 0$ we may hope to establish stability among the class of vanishing dissipation solutions.

To realize the above strategy, we fix a vanishing dissipation solution $U$ and let $U^{\kappa, a, \nu}$ be the sequence of Navier--Stokes--Fourier solutions converging to $U$ in the sense of \Cref{def:vanishing-diss-soln}.
We then first prove a stability bound between each solution $U^{\kappa, a, \nu}$ and our entropy wave $V$.
Note we now stand to gain weighted relative entropy due our entropy wave $V$ solving the Euler system~\eqref{eq:euler-system} while $U^{\kappa, a, \nu}$ solves the system~\eqref{eq:nsf-system} with a different pressure and dissipative terms. 
To compute the change of the weighted relative entropy~\eqref{eq:weighted-euler-example}, we modify it by replacing the shifted entropy wave $ V$ with a smoothed version which we denote $\bar V_\epsilon $.
This turns our attention to the regularized quantity
\begin{equation} \label{eq:weighted-euler-example-mollified} F_\epsilon(t) :=\int_{\T^3} \bar \theta_\epsilon (-\rho s) (U^{\kappa,a,\nu}| V_\epsilon^a)\,dx, \end{equation}
where $\bar \theta_\epsilon$ is the temperature of our regularized shifted front $V_\epsilon$.
By doing this, we are now multiplying the relative entropy by a smooth function.
In \Cref{lem:rel-entropy-shifts} we use the weak form of the entropy inequality and continuity equation to obtain a bound resembling
\begin{equation} \label{eq:proof-strat-bound}
    \begin{aligned}
        F_\epsilon(t) - F_\epsilon(0) &+ \int_0^t \int_{\T^3}  \kappa \frac{1+\theta^3}{\theta^2}|\nabla \theta|^2 + \nu |\nabla u|^2 \,dx\,d\tau \\
        \leq&\, C \kappa \int_0^t \int_{\T^3} \frac{1+\theta^3}{\theta} |\nabla \theta | |\nabla \bar\theta_{\delta,\epsilon}|\,dx\,d\tau   +  Ca\\
        &+ C\left|\int_0^t \int_{\T^3} \rho [\partial_t \bar V^a_{\epsilon} + u\cdot \nabla \bar V^a_{\epsilon}] \,dx\,d\tau\right|\\
        &+ C\left|\int_0^t \int_{\T^3} (- \rho s) [\partial_t \bar V^a_{\epsilon} + u\cdot \nabla \bar V^a_{\epsilon}] \,dx\,d\tau\right|,
    \end{aligned}
\end{equation}
where $\theta$ and $u$ are primitive variables associated to the Navier--Stokes--Fourier solution $U^{a,\kappa,\nu}$.
From this we deduce the bound~\eqref{eq:main-est} after showing that the right side of~\eqref{eq:proof-strat-bound} vanishes as $\kappa$, $a$, and $\nu$ converge to zero along a sequence of solutions $ U^{\kappa,a,\nu}\to U $ as in \Cref{def:vanishing-diss-soln} and proving that $F_\epsilon(t)$ is coercive to the quantities on the left side of~\eqref{eq:main-est} (see \Cref{lemma:coercive}).
We note that due to the mollification of $V$, the function $\bar V_\epsilon$ does not solve the transport equation with velocity $u$. 
As a result, this term appears as a commutator and, since $u \in L^2(0,T; H^1(\T^3))$, converges to zero in $L^2(0,T; L^2(\T^3))$. 
From this, the final integral converges to zero provided that $\rho,\rho s \in L^2(0,T;L^2(\T^3))$. 
This is given by the constitutive relations we select in \Cref{sec:con-rel} and \Cref{lemma:rho-s-integrability}, provided that $\gamma > 2$.
For further discussion of this limit and some explanation behind why we sequentially take $\kappa \to0,$ $a\to0$, then $\nu\to0$ in our vanishing dissipation limit, see the discussion in \Cref{sec:diss-limit}.

To determine the uniqueness of the entropy wave we appeal to the estimate~\eqref{eq:main-est}.

Since $\bar{\mathcal{E}}$ and $\bar u = 0$ are constants, we immediately obtain strong convergence of $m^n \to 0$ and $\mathcal{E}^n \to \bar {\mathcal{E}}$ in appropriate norms. 
For the density we note that $\mathcal{F}$ is not coercive.
This is because, unlike the energy and momentum, the entropy wave may have non-constant density.
Alone, it only shows that the weak limit $\rho$ takes values a.e. in an interval $[\rho_*, \rho^*]$. 
However, since $m_n \to 0$ and each $\rho_n, m_n$ distributionally satisfy the continuity equation, we can conclude that the weak limit $\rho$ satisfies 
$$  \partial_t \rho = 0.$$
Additionally, we can conclude by the convergence of the initial data in relative entropy~\eqref{eq:rel-ent-converge-init-data} that $\rho(0) = \bar \rho(0)$. 
Hence, the weak limit $\rho$ in fact coincides with the density of the entropy wave $\bar \rho$.

\vspace{1em}
The remainder of the paper is organized as follows. 
In \Cref{sec:prelim} we give details about the assumptions on the constitutive relations for the Navier--Stokes--Fourier system~\eqref{eq:nsf-system}, define our notion of admissible solutions for the Navier--Stokes--Fourier system, prove several a priori bounds verified by these solutions, and prove \Cref{thm:existence}. 
\Cref{sec:proof} is dedicated to the proof of \Cref{prop:main} where we construct the shift, find the form and coercive quantities of the relative entropy, and take the vanishing dissipation limit. 

\section{Preliminaries} \label{sec:prelim}
In this section, we give our assumptions on the constitutive relations of the Navier--Stokes--Fourier system, state an associated existence theorem for the Navier--Stokes--Fourier system, prove a priori bounds associated to these solutions, and establish \Cref{thm:existence}. 
For a physical motivation behind these constitutive relations, we refer the reader to \cite[Section 1.3]{feireislnovotny17}.

\subsection{Constitutive relations} \label{sec:con-rel}
In the following $\Lambda >0$ denotes some positive constant. 
We suppose that our viscous stress tensor satisfies Newton's rheological law, giving the form
$$ \S = \mu(\theta) \left(\nabla u + \nabla^T u - \frac{2}{3} \div(u) I\right) + \lambda(\theta) \div(u) I, $$
where $\mu,\lambda\in C^1( [0,\infty))$the shear and bulk viscosities, respectively.
We assume that they are globally Lipschitz and satisfy the bounds 
\begin{equation} \label{eq:visc-assum} \frac{1 + \theta}{\Lambda} \leq \mu(\theta) \leq \Lambda (1 + \theta) \quad \text{and}\quad 0 \leq \lambda(\theta) \leq \Lambda (1+\theta).
\end{equation} 
Further, we suppose that our heat flux $\mathbf{q}$ satisfies Fourier's law, 
\begin{equation}\label{eq:heat-flux} \mathbf{q}(\theta,\nabla \theta) = - K(\theta) \nabla \theta, \end{equation}
where the heat conductivity coefficient $K \in C^1( [0,\infty))$ enjoys the bounds 
\begin{equation}\label{eq:heat-cond-assum} \frac{1 }{\Lambda} (1 + \theta^3) \leq K(\theta) \leq  \Lambda (1+\theta^3). \end{equation}
As mentioned in the introduction, the pressure we consider takes the form $p=p_M+ap_R$ with the molecular and radiative pressures defined by
\begin{equation} \label{eq:pressure} p(\rho,\theta) = \theta^{\frac{\gamma}{\gamma - 1}} P\left( \frac{\rho}{\theta^{\frac{1}{\gamma - 1}}}\right) + \frac{a}{3} \theta^4 =: p_{M}(\rho,\theta) + ap_{R}(\theta), \end{equation}
where $a > 0$, $\gamma > 2$, and $P \in C^1([0,\infty))$ is an increasing function satisfying $P(0) = 0$. 
We further suppose that there exists a constant $\underline Z > 0$ such that $P(Z) = Z$ for all $Z \in [0, \underline Z]$. 

The associated specific internal energy has the form
\begin{equation} \label{eq:energy}
e(\rho,\theta) = \frac{1}{\gamma - 1} \frac{\theta^{\frac{\gamma}{\gamma - 1}}}{\rho} P\left( \frac{\rho}{\theta^{\frac{1}{\gamma - 1}}}\right) + a \frac{\theta^4}{\rho} =:e_M + ae_R,
\end{equation}
and the associated specific entropy is 
\begin{equation} \label{eq:entropy-def} s(\rho,\theta) = \mathcal{S}\left( \frac{\rho}{\theta^{\frac{1}{\gamma - 1}}}\right) + \frac{4a}{3} \frac{\theta^3}{\rho} =: s_M + as_R,\end{equation}
where $\mathcal{S}$ satisfies 
\begin{equation} \label{eq:Sdef}
    \mathcal{S}'(Z) = - \frac{1}{\gamma - 1} \frac{\gamma P(Z) - P'(Z)Z}{Z^2} < 0.
\end{equation}
These equations of state satisfy the Gibbs relation, 
\begin{equation} \label{eq:gibbs-relation}
    \theta \,ds = de - \frac{p}{\rho^2}\,d\rho.
\end{equation}
Let us assume that $s$ satisfies the third law of thermodynamics, which amounts to requiring that
\begin{equation} \label{eq:third-law}
    \lim_{Z\to \infty} \mathcal{S}(Z) = 0.
\end{equation}
We additionally assume that these functions satisfy the hypotheses of thermodynamic stability, 
\begin{equation}\label{eq:thermo-stab} 
    \frac{\partial p}{\partial \rho} >0,\quad \frac{\partial e}{\partial \theta} > 0\quad \text{for all}\quad\rho,\ \theta > 0.
\end{equation}
From the hypotheses of thermodynamic stability and taking that the molecular component of the specific heat at constant volume ($c_{v,M} = \frac{\partial e_M}{\partial \theta}$) is uniformly bounded, we have for all $Z > 0$
$$ 0 < \frac{\gamma P(Z) - P'(Z)Z}{Z} < C. $$
This inequality shows that $P(Z)/Z^\gamma$ is decreasing.
Finally, our last assumption is that
\begin{equation} \label{eq:Pinf-def} 0 < \lim_{Z\to \infty} \frac{P(Z)}{Z^\gamma} =: P_\infty.\end{equation}
 Note that from~\eqref{eq:Sdef} since $P(Z) = Z$ for all $Z \leq \underbar Z$ the specific molecular entropy $s_M$ takes the form, for some constant $C$,
\begin{equation} \label{eq:S-form}
    s_M(Z)= C - \ln\left( Z \right),  \quad \hbox{for all } Z < \underbar{Z}.
\end{equation}

\subsection{Admissible solutions to the Navier--Stokes--Fourier system}
In this section we state a notion of admissible solutions to the system~\eqref{eq:nsf-system} and an associated existence theorem associated to our constitutive relations in \Cref{sec:con-rel}. These results are from \cite{novotnyexist13,feireislnovotny17}.
We first define a notion of weak solution to the system~\eqref{eq:nsf-system}.
\begin{definition} \label{def:solution}
    A trio $(\rho, u,\theta)$ is a weak solution to the Navier--Stokes--Fourier system~\eqref{eq:nsf-system} if 
    \begin{enumerate}
        \item The density and temperature satisfy $\rho(t,x) \geq 0,\ \theta(t,x) > 0$ a.e. in $(0,T)\times \T^3$,
            $\rho \in L^\infty(0,T; L^{\gamma}(\T^3))$, $\rho u \in L^\infty (0,T; L^{\frac{2\gamma}{\gamma + 1}}(\T^3;\R^3) )$,
            $\rho |u|^2 \in L^\infty(0,T; L^1(\T^3))$, $\theta \in$ \\$ L^\infty(0,T; L^4(\T^3))\cap L^2(0,T; H^1(\T^3))$, 
            $\nabla \log(\theta) \in L^2(0,T; L^2(\T^3;\R^3))$, and \\$u \in L^2(0,T; H^1(\T^3)) $. 

        \item $\rho \in C_{weak}([0,T]; L^\gamma(\T^3))$ and weakly satisfies the continuity equation,
        \begin{equation} \label{eq:nsf-cont}
            \left.\int_{\T^3} \rho \phi \,dx \right|_0^\tau = \int_0^\tau \int_{\T^3} \rho \partial_t \phi + \rho u \cdot \nabla \phi\,dx\,dt
        \end{equation}
        for all $\tau \in [0,T]$ with $\phi \in C^1([0,T] \times \T^3).$

        \item $\rho u \in C_{weak}([0,T]; L^{\frac{2\gamma}{\gamma + 1}}(\T^3;\R^3))$ and weakly satisfies the momentum equation, 
        \begin{equation}
            \begin{aligned}
            \left.\int_{\T^3} \rho u\cdot \phi \,dx \right|_0^\tau =& \int_0^\tau \int_{\T^3} \rho u\partial_t \phi + \rho u\otimes u \cdot \nabla \phi\,dx\,dt \\ 
            &+ \int_0^\tau \int_{\T^3} p(\rho, \theta)\div(\phi) - \nu\,\S(\theta,\nabla u):\nabla \phi \,dx\,dt
            \end{aligned}
        \end{equation}
        for all $\tau \in [0,T]$ with $\phi \in C^1([0,T] \times \T^3; \R^3)$.

        \item We have a distributional entropy inequality 
        \begin{equation} \label{eq:nsf-entropy-inequality}
            \begin{aligned}
            -\left.\int_{\T^3} \rho s(\rho,\theta) \phi \,dx \right|_0^\tau &+ \int_0^\tau \int_{\T^3} \frac{\phi}{\theta} \left(\nu\,\S(\theta,\nabla u): \nabla u - \kappa\frac{\mathbf{q}(\theta,\nabla \theta)\cdot \nabla \theta}{\theta}\right) \,dx\,dt \\
            \leq& -\int_0^\tau \int_{\T^3} \rho s(\rho, \theta) \partial_t \phi + \rho s(\rho,\theta) u\cdot \nabla \phi + \kappa\frac{\mathbf{q}(\theta,\nabla \theta)\cdot \nabla \phi}{\theta} \,dx\,dt
            \end{aligned}
        \end{equation}
        for a.e. $\tau \in [0,T]$ with $\phi \in C^1([0,T] \times \T^3),$ $\phi \geq 0$. 

        \item We have conservation of energy,
        \begin{equation} \label{eq:E-con}
            \int_{\T^3} \left[\frac{1}{2} \rho |u|^2 + \rho e(\rho,\theta)\right](\tau)\,dx = \int_{\T^3} \left[\frac{1}{2} \rho |u|^2 + \rho e(\rho,\theta)\right](0)\,dx
        \end{equation}
        for a.e. $\tau \in [0,T]$.
    \end{enumerate}
\end{definition}

\begin{proposition}[{\cite{feireislnovotny17,novotnyexist13}}] \label{prop:nsf-exist}
Fix the system~\eqref{eq:nsf-system} satisfying the constitutive relations of \Cref{sec:con-rel} and an initial data $U_0 =(\rho_0,  m_0, \mathcal{E}_0)$ for which there exists $\theta_0$ where $\mathcal{E}_0 = \frac{|m_0|^2}{2\rho_0} + \rho_0 e(\rho_0,\theta_0)$.
If $U_0$ satisfies that
\begin{align*} 
    \rho_0 \geq 0,\quad\rho_0 \in L^{\gamma}(\T^3),\quad &\int \rho_0\,dx > 0,\\
    m_0 = 0\ \text{a.e. on the set }&\{x\in \T^3\mid\rho_0(x) = 0\} \\
    \quad\theta_0 > 0\text{ a.e. in }\T^3,\ &\rho_0s(\rho_0,\theta_0) \in L^1(\T^3),\\
    \int \mathcal{E}_0\,dx = \int &\frac{|m_0|^2}{2\rho_0} + \rho_0 e(\rho_0,\theta_0)\,dx<\infty,
\end{align*}
then there exists a solution $(\rho, u, \theta)$ to the Navier--Stokes--Fourier system~\eqref{eq:nsf-system} with this initial data in the sense of \Cref{def:solution}. 
\end{proposition}

\subsection{A priori bounds on admissible solutions} \label{sec:prelim-bounds}
Suppose $(\rho, u, \theta)$ is an admissible solution to the system~\eqref{eq:nsf-system} in the sense of Definition \ref{def:solution} for times $0 \leq t \leq T$.
{In this section, we are interested in finding a priori bounds which allow us to take terms on the right side of inequality~\eqref{eq:proof-strat-bound} to zero in the vanishing dissipation limit. 
In addition to this, these bounds will be used to obtain uniform (in $\kappa, \nu, a$) bounds of $\rho$, $m$, and $\mathcal{E}$ to provide the weak limits of solutions $(\rho,m,\mathcal{E})$. }

We first note that by thermodynamic stability~\eqref{eq:thermo-stab} and~\eqref{eq:Pinf-def} we have 
\begin{equation} \label{eq:rho-gamma}
    \begin{aligned}
        \rho e_M(\rho,\theta) &= \frac{1}{\gamma - 1} \theta^{\gamma/(\gamma -1)}P(\rho/\theta^{1/(\gamma -1)}) \\
        &\geq \lim_{\theta \to 0}\frac{1}{\gamma - 1} \theta^{\gamma/(\gamma -1)}P(\rho/\theta^{1/(\gamma -1)}) = \frac{1}{\gamma-1}P_\infty \rho^\gamma.
    \end{aligned}
\end{equation}
Furthermore, from the conservation of energy~\eqref{eq:E-con} and constitutive relations~\eqref{eq:energy} we find 
\begin{multline} \label{eq:bounds-from-energy}
        a \norm{\theta}_{L^\infty(0,T;L^4(\T^3))}^4+\dfrac{1}{\gamma-1}\norm{p_M}_{L^\infty(0,T;L^1(\T^3))}  + {\dfrac{1}{2}} \norm{\rho|u|^2}_{L^\infty(0,T;L^1(\T^3))} \\\leq C\int_{\T^3} \frac{1}{2} \rho(x, 0)|u(x, 0)|^2+(\rho e)(x, 0)\,dx
\end{multline}
for some $C > 0$, independent of $a$, $\kappa$, and $\nu$. 
In particular, combining this with~\eqref{eq:rho-gamma} and $\rho e_M = \tfrac{1}{\gamma - 1}p_M$ we have 
\begin{equation} \label{eq:rho-integability}
    \norm{\rho}_{L^\infty(0,T; L^\gamma(\T^3))}^\gamma \leq C  \int_{\T^3} \frac{1}{2} \rho(x, 0)|u(x, 0)|^2+(\rho e)(x, 0)\,dx.
\end{equation}

Next, testing our entropy inequality~\eqref{eq:nsf-entropy-inequality} with the constant function $\phi = 1$ we find that
\begin{multline} \label{eq:diss-bound-from-entropy}
    \int_0^T \int_{\T^3} \frac{1}{\theta} \left(\nu\,\S(\theta,\nabla u):\nabla u - \kappa \frac{\mathbf{q}(\theta,\nabla\theta)\cdot \nabla \theta}{\theta}\right)\,dx \,dt\\ \leq \left.\int_{\T^3} \rho s(\rho,\theta)\,dx \right|_{t = T} - \left.\int_{\T^3}  \rho s(\rho,\theta)\,dx \right|_{t = 0} . 
\end{multline}
To bound the right hand side, we note that $\rho s_R \in L^\infty(0,T;L^{4/3}(\T^3))$ by~\eqref{eq:bounds-from-energy}.
For the molecular entropy we find a bound based on the integrability of $\rho$. 
Next, we show that the form of our pressure $p_M$ and the third law of thermodynamics~\eqref{eq:third-law} provide a bound on $\rho s_M$ in $L^\infty(0,T; L^2(\T^3))$. 
\begin{lemma} \label{lemma:rho-s-integrability}
    Suppose $(\rho, u,\theta)$ is an admissible solution (\Cref{def:solution}) to a Navier--Stokes--Fourier system satisfying the constitutive relations of \Cref{sec:con-rel}.
    We then have the bound
    \begin{equation}
        \norm{\rho s_M(\rho,\theta)}_{L^\infty(0,T; L^2(\T^3))} \leq C,
    \end{equation}
    where $C$ depends on the initial total energy $\norm{\frac{1}{2}\rho_0 |u_0|^2 + \rho_0e(\rho_0,\theta_0)}_{L^1(\T^3)}$ and $\gamma$.
\end{lemma}
\begin{proof}
    Fixing a $t \in [0,T]$ we define $A_1 = \{x\mid \rho(t,x) > \underbar Z\theta(t,x)^{1/(\gamma - 1)} \}$ and $A_2 = \T^3 \setminus A_1$, where $\underline Z$ is as defined in \Cref{sec:con-rel}. 
    We recall that since $s_M = \mathcal{S}(\rho/\theta^{1/(\gamma - 1)})$ satisfies the third law of thermodynamics \eqref{eq:third-law}, there exists a $C>0$ such that 
    \begin{equation} \label{eq:third-law-usage}\sup_{x\in A_1} |s_M(\rho,\theta)| \leq \sup_{Z > \underline{Z}} \mathcal{S}(Z) < C.\end{equation}
    Integrating $(\rho s_M)^2$ and using \eqref{eq:third-law-usage} we find
    \begin{equation} \label{eq:entropy int first} \int_{\T^3}(\rho s_M)^2\,dx = \int_{A_1} (\rho s_M)^2 \,dx + \int_{A_2 } (\rho s_M)^2 \,dx \leq C^2 \int_{A_1} \rho^2 \,dx + \int_{A_2} (\rho s_M)^2 \,dx, \end{equation}
    where we recall that the first term is bounded by~\eqref{eq:rho-integability}. 
    Substituting~\eqref{eq:S-form} into the second term of~\eqref{eq:entropy int first} we have by Young's inequality,
    \begin{align}
        \int_{A_2} (\rho s_M)^2\,dx &\leq C\int_{A_2} \rho^2 \left| \log\left(\frac{\rho}{\theta^{1/(\gamma - 1)}  }\right) \right|^2 \,dx + C \int_{{A_2}} \rho^2\,dx.\nonumber\\ 
        \intertext{To handle the first term of the above expression we use the identity $|\ln(z)| \leq C z^{-\alpha}$ for all $z \leq \bar Z$ and H\"older inequality to obtain}
        \int_{A_2} \rho^2 \left| \log\left(\frac{\rho}{\theta^{\frac{1}{\gamma - 1}} }\right) \right|^2\,dx&\leq C  \int_{A_2}  \rho^{2(1-\alpha)} \theta^{2\frac{\alpha}{\gamma - 1}}\,dx \nonumber\\
        &=C\int_{A_2} \rho^{2\left(1-\frac{\gamma\alpha}{\gamma - 1}\right) } (\rho\theta)^{2\frac{\alpha}{\gamma - 1}} \,dx \nonumber\\
        &\leq C\left(\int_{A_2} \rho \theta \,dx \right)^{2\frac{\alpha}{\gamma - 1}} \left( \int_{A_2} \rho^{2\left(1-\frac{\gamma\alpha}{\gamma - 1}\right)/\left(1 - 2\frac{\alpha}{\gamma - 1}\right) }\,dx \right)^{ 1 - 2\frac{\alpha}{\gamma - 1}}. \label{eq:lem-ent-int-final}
    \end{align}
    Now, by the form of our pressure~\eqref{eq:pressure} we have
    $$p_M(\rho,\theta) = \rho\theta \quad\text{for all}\quad x\in A_2$$ 
    allowing us to bound the first integral by~\eqref{eq:bounds-from-energy}.
    For the second integral, we note that if $\alpha$, $\gamma$ satisfy the inequality 
    \begin{equation} \label{eq:lem-ent-int-gamma} 
        2\frac{1-\frac{\gamma\alpha}{\gamma - 1}}{1 - 2\frac{\alpha}{\gamma - 1}} \leq \gamma
        \end{equation}
    then we can use the bound~\eqref{eq:rho-integability}, as $\T^3$ has finite measure.
    We note that as $\alpha \to 0$ the left side of \eqref{eq:lem-ent-int-gamma} converges to $2$. 
    Therefore, since $\gamma > 2$ there exists $\alpha$ sufficiently small such that the inequality~\eqref{eq:lem-ent-int-gamma} is satisfied and \eqref{eq:lem-ent-int-final} is bounded by constants depending on $\gamma$ and the initial data. 
\end{proof}
Now that $\rho s$ is bounded, recalling~\eqref{eq:diss-bound-from-entropy} and our assumptions on the viscosity coefficients~\eqref{eq:visc-assum} we have the bound
\begin{equation} \label{eq:bounds-from-diss}
    \begin{aligned}
        C &\geq \nu \norm{\nabla u + \nabla^T u - \frac{2}{3} \div(u) I}_{L^2( 0,T; L^2(\T^3))}^2 + \kappa \norm{ \frac{\sqrt{K(\theta)}}{\theta} \nabla \theta}_{L^2( 0,T; L^2(\T^3))}^2,
    \end{aligned}
\end{equation}
where $C$ depends only on our constitutive relations and initial data.

To deduce a bound on $u$ from~\eqref{eq:bounds-from-diss} we recall the generalized Korn--Poincar\'e inequality.
\begin{lemma}[{\cite[Proposition 2.1]{feireislnovotny17}}] \label{lemma:gen-korn}
    Suppose that $r$ is a non-negative function such that there exists $C > 1$ for which
    $$ \frac{1}{C} \leq \int_{\T^3} r \,dx\quad \text{and}\quad \int_{\T^3} r^\gamma \,dx\leq C,$$
    where $\gamma >1$.
    Then for any $v \in W^{1,p}(\T^3)$ with $p > 1$ there exists a $K > 0$ giving the bound 
    $$ \norm{v}_{W^{1,p}(\T^3)} \leq K\left(\norm{\nabla v + \nabla^T v - \frac{2}{3} \textup{div}(v) I}_{L^p(\T^3; \R^{3\times 3})} + \int_{\T^3} r|v|\,dx\right) $$
    where the constant $K$ only depends on $p$, $\gamma$, and $C$.
\end{lemma}
Since $\rho$ satisfies conservation of mass, using conservation of energy~\eqref{eq:bounds-from-energy}, and the dissipation bound~\eqref{eq:bounds-from-diss} we can apply \Cref{lemma:gen-korn} with $v = u$, $r = \rho$, $p = 2$, and the same $\gamma$ as in our pressure to deduce that
\begin{equation} \label{eq:u-H1-bound}
    \norm{u}_{L^2(0,T; H^1(\T^3)) }^2 \leq C \left( \frac{1}{\nu} + 1\right)
\end{equation}
for some $C > 0$, independent of $a$, $\kappa$, and $\nu$.


\subsection{Proof of \Cref{thm:existence}}
In this section we establish \Cref{thm:existence} using the a priori bounds deduced in this section and the existence result for the Navier--Stokes--Fourier system \Cref{prop:nsf-exist}. 

We note that $\rho_0 \in L^\gamma(\T^3)$ by the bound~\eqref{eq:rho-gamma}.
This in addition to the other assumptions of \Cref{thm:existence} allows us to apply \Cref{prop:nsf-exist} for any $\kappa,a,\nu>0$ to obtain a weak solution $ (\rho^{\kappa,a,\nu},u^{\kappa,a,\nu},\theta^{\kappa,a,\nu}) $ in the sense of \Cref{def:solution}.
Next, noting that by H\"older's inequality
$$ \norm{\rho u}_{L^\infty(0,T;L^{2\gamma / (\gamma + 1)}} \leq \norm{\sqrt{\rho}}_{L^\infty(0,T;L^{2\gamma}(\T^3))} \norm{\sqrt{\rho}u}_{L^\infty(0,T; L^2(\T^3))} $$
and the estimates~\eqref{eq:bounds-from-energy} and~\eqref{eq:rho-integability} we have
$$ \norm{\rho}_{L^\infty(0,T; L^\gamma(\T^3))} + \norm{\sqrt{\rho}u}_{L^\infty(0,T; L^2(\T^3))}\leq C(1+ a\norm{\theta_0}_{L^4(\T^3)}^4) $$
where $C$ depends on the initial data and constitutive relations of \Cref{sec:con-rel}.
Furthermore, by the conservation of energy~\eqref{eq:E-con} we find 
$$ \norm{\mathcal{E}}_{L^\infty(0,T;L^1(\T^3))} \leq C (1 + a\norm{\theta_0}^4), $$
where $C$ is again independent of $\kappa$, $a$, and $\nu$.
Given these bounds, we can pass to the limit in $\kappa$ to find a subsequence along which 
\begin{align*}
    \rho^{\kappa,a,\nu} &\wsto \rho^{a,\nu} \quad \text{weakly$-\ast$ in}\quad L^\infty(0,T; L^\gamma(\T^3)), \\
    \rho^{\kappa,a,\nu}u^{\kappa,a,\nu} &\wsto m^{a,\nu} \quad \text{weakly$-\ast$ in}\quad L^\infty(0,T; L^{2\gamma/(\gamma + 1)}(\T^3)), \\
    \mathcal{E}^{\kappa,a,\nu} &\wsto \mathcal{E}^{a,\nu}\quad\text{weakly$-\ast$ in}\quad L^\infty(0,T; \mathcal{M}(\T^3)),
\end{align*}
by Banach Alaoglu.
The above bounds are maintained in this limit, so we may repeat this process to obtain $U = (\rho, m, \mathcal{E})$ which is a vanishing dissipation solution in the sense of \Cref{def:vanishing-diss-soln}.


\section{Proof of \Cref{prop:main}} \label{sec:proof}
In this section we prove \Cref{prop:main}. It is organized by first defining a notion of shift, mirroring the considerations made for conservation laws in 1-d. This is followed by computing the form of the relative entropy and its coercive quantities. Finally, we compute the weighted relative entropy inequality and take the vanishing dissipation limit to deduce the desired bound~\eqref{eq:main-est} on vanishing dissipation limits and the uniqueness result. 

\subsection{Constructing the shift} \label{sec:shift}
We fix an admissible solution $(\rho,u,\theta)$ in the sense of \Cref{def:solution} to the Navier--Stokes--Fourier system. 
Recall by the assumption~\eqref{eq:V-form}, our entropy wave takes the form $V(x) = (\bar \rho(x), 0,\bar \theta(x))$ in state variables with $\bar \rho_0$ and $\bar \theta_0$ both bounded above and below.
To construct our shifted profile relative to a Navier--Stokes--Fourier solution $(\rho, u, \theta)$ we mollify the velocity field. Let us then introduce $u_\delta$, which, for a given $\delta \in (0,1)$, is defined by 
$$u_\delta := u\ast \varphi_\delta,$$
where $\varphi \in C^\infty(\R^3)$, $\varphi \geq 0$, $\int \varphi\,dx = 1$, $\text{supp}(\varphi) \subset B_1(0)$ (where $B_1(0)$ is the open unit ball centered at $0$), and
\begin{equation} \label{eq:def-mol} 
    \varphi_\delta(x) = \begin{cases} 
        \frac{1}{\delta^3} \varphi\left(\frac{x}{\delta}\right)\quad & |x|\leq \delta\\
        0 & |x| \geq \delta
    \end{cases}
\end{equation}
over the cube $[-1,1]^3$ and periodically extended so that $\varphi_\delta \in C^\infty(\T^3)$.

We now let $\bar \theta_\delta$ solve the following transport equation with velocity $u_\delta$,
\begin{equation}\label{eq:delta-shift-pde} \begin{cases}
    \begin{aligned}
        \partial_t \bar \theta_\delta + u_\delta\cdot \nabla \bar \theta_\delta &= 0, \\
        \bar\theta_\delta(0,\cdot) &= \bar\theta_0,
    \end{aligned}
\end{cases} 
\end{equation}
where the initial data $\bar \theta_0$ corresponds to the temperature of the entropy wave $V$~\eqref{eq:V-form}.
We note that since $u_\delta$ is smooth, the above problem admits a unique weak solution by the method of characteristics. 
To ensure that these solutions are sufficiently regular, we add the additional mollification for $\epsilon \in (0,1)$
$$\bar \theta_{\delta,\epsilon} := \bar \theta_\delta \ast \varphi_\epsilon,$$
which we will take to be our shifted temperature.
To define our shifted density $\rho_{\delta,\epsilon}$ we recall that along the entropy wave $$p_M(\bar \rho, \bar \theta) \equiv \bar p$$ for some positive constant $\bar p$. 
Using the hypothesis of thermodynamic stability~\eqref{eq:thermo-stab} we know that $\partial_\rho p_M(\rho,\theta) > 0$ and additionally for any fixed $\theta > 0$ the range of $p_M(\cdot, \theta)$ is all of $[0,\infty)$.
This allows us to deduce a unique shifted density $\bar \rho_{\delta,\epsilon}$ such that $\bar \rho_{\delta,\epsilon}$ and $\bar \theta_{\delta,\epsilon}$ satisfy
$$ p_M(\bar \rho_{\delta,\epsilon} , \bar \theta_{\delta,\epsilon}) = \bar p.$$
Ensuring that this identity is maintained at the level of the smoothed density and temperature will be helpful in the future analysis. 
 
We now denote our shifted entropy wave in conservative variables by
$$ V_{\delta,\epsilon}(t,x) = (\bar \rho_{\delta,\epsilon}(t,x), 0, \bar \rho_{\delta,\epsilon}(t,x) e_M(\bar \rho_{\delta,\epsilon} , \bar \theta_{\delta,\epsilon})),$$
where $\bar \rho_\delta$ is defined analogously to $\bar \rho_{\delta,\epsilon}$.
We will also need a ``lift'' of this wave to account for the dependence of our pressure $p$ of the Navier--Stokes--Fourier system on the parameter $a > 0$.
We define
$$ V_{\delta,\epsilon}^a(t,x) = V_{\delta,\epsilon}(t,x) + (0,0, a \bar \rho_{\delta,\epsilon}(t,x) e_R(\bar\rho_{\delta,\epsilon}(t,x),\bar \theta_{\delta,\epsilon})).$$

Finally, we note that since $\bar \rho_0, \bar \theta_0 \in L^\infty(\T^3)$ our shifted primitive variables have the uniform bounds 
\beq \label{eq:psi-bounded} C^{-1} \leq \bar \theta_{\delta},\bar \theta_{\delta,\epsilon}, \bar \rho_{\delta},\bar \rho_{\delta,\epsilon}\leq C \eeq
for some $C > 0$.
Additionally, due to smoothness of $p_M$ and~\eqref{eq:psi-bounded} we similarly have 
\begin{equation} \label{eq:V-de-Linf-bound} |V_{\delta,\epsilon}| \leq C,\quad |V^a_{\delta,\epsilon}| \leq C(1+a) \end{equation}
for some $C > 0$, where in this case $|\cdot|$ is Euclidean distance in $\R^5$ and $V_{\delta,\epsilon},\ V^a_{\delta,\epsilon}$ are written in the conservative coordinates. 


\subsection{Weighted relative entropy inequality}
In this section we establish the form of the relative entropy. 
\begin{lemma} \label{lem:rel-entropy-eval}
    Consider two constant states $U = (\rho, m, \mathcal{E}),\ V = (\bar \rho, 0,\bar {\mathcal{E}})$ {where the internal energy of both are determined via the relation $e = e_M + ae_R$} (without confusion, we will continue to disambiguate the other thermodynamic variables by the use of the bar for those corresponding to the state $V$ and no bar for those corresponding to $U$). 
    The relative entropy has the form
    $$\bar \theta (-\rho s)(U|V) = \mathcal{E}+ \bar p -\bar \theta \rho s - \rho \bar c ,$$
    where $\bar c = \bar e - \bar \theta \bar s + p(\bar \rho,\bar \theta) /\bar \rho$. 
\end{lemma}
\begin{proof}
    Using \eqref{eq:def_relative_entropy} with $\eta(U) = (-\rho s)(U),$ we have
    \begin{equation} \label{eq:rel ent def} (-\rho s)(U|V) = (-\rho s)(U) - (-\rho s)(V) -d(-\rho s)(V)(U-V). \end{equation}
    Using the Gibbs relation~\eqref{eq:gibbs-relation} and that 
    $$e=\frac{{\mathcal{E}}}{\rho} - \frac{|u|^2}{2},$$
    we obtain
    \begin{equation} \label{eq:gibbs-rho}
        \bar \theta\, d(-\rho s)(V) = \left(\bar \theta(-s)(V) + \frac{p(\bar \rho,\bar \theta) + \bar{\mathcal{E}}}{\bar \rho}\right)\, d\rho - d\mathcal{E}.
    \end{equation}
    Using~\eqref{eq:rel ent def} we find the relation  
    \begin{align*}
        \bar \theta(-\rho s)(U|V) &= \bar \theta(-\rho s)(U) - \bar \theta (-\rho s)(V) - \bar \theta\, d(-\rho s)(V)(U-V) \\
        &= -\bar \theta \rho s + \bar \theta\bar \rho \bar s - \left(-\bar \theta \bar s + \frac{p(\bar \rho,\bar \theta) + \bar{\mathcal{E}}}{\bar \rho}\right)(\rho-\bar \rho) + (\mathcal{E} - \bar{\mathcal{E}}) \\
        &= -\bar\theta \rho s - \left( -\bar \theta \bar s+ \frac{p(\bar \rho,\bar \theta)  }{\bar \rho}+\bar e \right) \rho +p(\bar \rho,\bar \theta) + \mathcal{E}.
    \end{align*}
    This completes the proof.

\end{proof}
In the future, we will compare to the shifted wave $V^a_{\delta,\epsilon}$ rather than to a constant state $V$. 
In this case, we denote by $\bar \theta_{\delta,\epsilon}(t,x)$ the temperature of the state $V^a_{\delta,\epsilon}(t,x)$, $\bar c_{\delta,\epsilon}(t,x) = \bar e_{\delta,\epsilon} - \bar \theta_{\delta,\epsilon} \bar s_{\delta,\epsilon} + p(\bar \rho_{\delta,\epsilon},\bar \theta_{\delta,\epsilon}) /\bar \rho_{\delta,\epsilon}$, and $\bar p_{\delta,\epsilon} = p(\rho_{\delta,\epsilon},\theta_{\delta,\epsilon})$, where $\bar e_{\delta,\epsilon}, \bar \theta_{\delta,\epsilon}, \bar s_{\delta,\epsilon}, \bar \rho_{\delta,\epsilon}$ are again values of the thermodynamic quantities at the state $V^a_{\delta,\epsilon}(t,x)$.

We now show that several quantities are coercive to the relative entropy. 
This bound follows from the hypothesis of thermodynamic stability~\eqref{eq:thermo-stab}, which grants that both $(-\rho s)$ is a convex function of the conservative variables $(\rho, m, \mathcal{E})$ \cite{maria-convex}.
\begin{lemma} \label{lemma:coercive}   
    There exists a $C> 0$ such that the following holds. For any state in phase space $U = (\rho, m, \mathcal{E})$ and $a \in [0,1]$ the relative entropy controls the following quantities,
    \begin{equation}\label{eq:rel-co} \mathcal{G}(m) + \mathcal{F}(\rho) + \mathcal{G}( \mathcal{E} - \bar {\mathcal{E}}^a ) \leq C\bar \theta(V_{\delta,\epsilon}(t,x)) (-\rho s)(U|V^a_{\delta,\epsilon}(t,x)) \quad\text{for all}\quad t\geq 0,\,x \in \T^3,\end{equation}
    where $\bar{\mathcal{E}}^a = \bar \rho e_M(\bar \rho,\bar \theta) + a\bar \rho e_R(\bar \rho,\bar \theta)$ is the total energy of the shifted state $V_{\delta,\epsilon}^a$.
    In particular, the constant $C$ is \emph{uniform} in the parameters $\kappa,\, a,$ and $\nu$. 
\end{lemma}

To clarify notation, in the following proof we denote the partial derivative $\partial_x f(x,y) = \frac{\partial f}{\partial x}(x,y)$ as the derivative obtained by varying $x$ while keeping the value of $y$ fixed. 
In particular, this distinguishes $\partial_v s(v,e)$ from $\partial_v s(v,\theta)$.
\begin{proof}
    Following the steps of \cite[Section 5]{maria-convex}, we will begin by showing that $-s = -(s_M + as_R)$ has strictly positive definite hessian in the variables $(v, e)$, where $v = \frac{1}{\rho}$ is the specific volume and $0 < v$ and $ \frac{P_\infty}{\gamma -1} v^{1-\gamma}< e$, the natural bounds away from vacuum and $\theta = 0$ (see~\eqref{eq:rho-gamma}).
    From this, the strict convexity of $-\rho s$ in the conservative variables $(\rho, m , \mathcal{E})$ follows via a change of coordinates (see \cite[Section 6]{maria-convex} for further details).

    To begin, we note that by Sylvester's criterion 
    \begin{equation} \label{eq:entropy-sylvester} \partial_e^2 s(v,e) < 0 \quad \text{and}\quad (\partial_v^2 s(v,e)) (\partial_e^2 s(v,e)) - \left( \partial_v \partial_e  s(v,e)\right)^2 > 0 \end{equation}
    is equivalent to the negative definiteness of the hessian of $s$. 
    By the Gibbs relation~\eqref{eq:gibbs-relation} we know that $\frac{1}{\theta(v,e)} = \frac{\partial s}{\partial e}(v,e)$ and by differentiating in $e$ and using~\eqref{eq:thermo-stab} we find 
    \begin{equation} \label{eq:entropy-coer-first-term-neg} \frac{\partial^2 s}{\partial e^2}(v,e) = -\frac{1}{\theta^2} \frac{\partial \theta}{\partial e}(v,e) < 0. \end{equation}
    Next, we write $\theta(v,\theta) = \theta(v, e(v,\theta))$ and differentiating in $v$ we find $0 = \partial_v \theta(v,e) + \partial_e \theta(v,e) \partial_v e(v,\theta)$. 
    By the relation $\theta(v,e)^{-1} = \frac{\partial s}{\partial e}(v,e)$ this is equivalent to
    $$ \frac{\partial e}{\partial v}(v,\theta) = -\frac{\partial_v \theta(v,e)}{\partial_e \theta(v,e)} = - \frac{\partial_e \partial_v s(v,e)}{\partial_e^2 s(v,e)}. $$
    Using again~\eqref{eq:thermo-stab} and $p = \theta \dfrac{ds}{dv}$ we compute 
    \begin{align*}
        0 > \frac{\partial p}{\partial v}(v,\theta) &=  \partial_v\left(\theta \frac{\partial s}{\partial v}(v,e(v,\theta))\right) \\
        &= \theta \left[\frac{\partial^2s}{\partial v^2}(v,e) + \frac{\partial^2s}{\partial v\partial e}(v,e)\frac{\partial e}{\partial v}(v,\theta) \right] \\
        &= \theta \left[ \frac{\partial^2s}{\partial v^2}(v,e) - \left(\frac{\partial^2s}{\partial v\partial e}(v,e)\right)^2 \left(\frac{\partial^2 s}{\partial e}(v,e)\right)^{-1} \right].
    \end{align*}
    Using the negativity~\eqref{eq:entropy-coer-first-term-neg} and $\theta > 0$ we see the above bound establishes the determinant condition of Sylvester's criterion~\eqref{eq:entropy-sylvester}, and hence also the negative definiteness of the hessian of $s$. 


    Now that we have established the strict convexity of the entropy, we prove the coercivity bound~\eqref{eq:rel-co}.
    Define the set $$W = \left[\rho_\ast, \rho^\ast\right] \times \{0\}\times  \left[\inf_{x\in\T^3} \bar \rho e_M(\bar \rho,\bar \theta), \sup_{x\in \T^3} \bar \rho e_M(\bar \rho,\bar \theta) + \sup_{x\in \T^3} \bar \theta^4\right] ,$$ 
    where $\rho_*,\rho^*$ are defined in~\eqref{eq:density-bounds}.
    This is a bounded set away from vacuum containing all states attained by $V^a_{\delta,\epsilon}$ for $a \in [0,1]$. 
    Note that there exists $r > 0$ such that 
    \begin{equation} \label{eq:math-cal-G-def} 
        K = \{U: d(U,\overline{W}) \leq r\}\subset\subset \mathcal{P}:=\left\{(\rho, m, \mathcal{E})\ \middle|\ \rho > 0\text{ and } \mathcal{E} - \frac{|m|^2}{2\rho} > \frac{P_\infty}{\gamma - 1} \rho^\gamma\right\},
    \end{equation}
    where $P_\infty$ is as in \Cref{eq:rho-gamma} and $\mathcal{P}$ is the conservative phase space corresponding to the constitutive relations of \Cref{sec:con-rel}.
    Considering the compact set $K\times [0,1]$ and recalling that $s = s_M + as_R$, the map $$(V,a) \mapsto \min \{\lambda \in \R: \lambda \text{ is an eigenvalue of } \nabla^2(-\rho s)(V)\}$$ is continuous on the compact set $ K\times [0,1] $. 
    Since $\nabla^2(-\rho s)$ is positive definite everywhere on $K$ this map has a strictly positive minimum on the set. 
    So, there exists $c > 0$ such that 
    $$ \nabla^2 (-\rho s)(V) \xi\cdot \xi \geq c\xi\cdot \xi\quad \forall \xi \in \R^5,\ V \in K,\ a\in[0,1], $$
    establishing a uniform convexity in a neighborhood of $\bar W$.

    Now, following the steps of \cite[Lemma 2.2]{kvy21}, we fix $a \geq 0$ and use this uniform convexity to establish the inequality~\eqref{eq:rel-co}.
    Fixing $V \in W$ and a direction $\xi \in \R^5 $ with $|\xi| = 1$ we use the relative entropy~\eqref{eq:def_relative_entropy} with $\eta = -\rho s$ to define the convex function $f(t) = \eta( V+t\xi | V)$ for $t \in [0,T)$, where $T$ is the largest value for which when $t < T$ we have $V+t\xi \in \mathcal{P}$.
    We note that $\mathcal{P}$ is a convex open set, so any state $U \in \mathcal{P}$ can be reached by such a ray. 
    Computing the derivatives, we find $f'(t) = (\nabla \eta(V+t\xi) - \nabla \eta(V))\cdot \xi$ which is zero at $t = 0$. 
    Furthermore, $f''(t) = \nabla^2 \eta(V+t\xi) \xi\cdot \xi$ and for all $t \leq r$ we have $V+t\xi \in K$, hence 
    $$ f''(t) \geq c \quad \forall t \in [0,r].$$
    Integrating this, we find 
    $$ f(t) = \int_0^t \int_0^s \nabla^2 \eta(V+q\xi) \xi\cdot \xi \,dq\,ds \geq \frac{1}{2}ct^2 \quad \forall t \in [0,r].$$
    Thanks to $\nabla^2 \eta $ being positive definite over $\mathcal{P}$, we have $f'' > 0$ for $t \in [0,T)$.
    This indicates that $f'$ is an increasing function for all $V$ and $\xi$.
    In particular, since $f'(0)=0$ we must have $f'(r/2) > 0$ and, furthermore, since $K\times [0,1]$ is compact there exists $c > 0$ such that 
    $$ f'(r/2) > c \quad \forall V \in W,\ |\xi| = 1,\ a \in[0,1].$$ 
    We then observe
    $$ f(t) = \int_0^t f'(s)\,ds = \int_{r/2}^t f'(s)\,ds + f(r/2) \geq c(t-r/2) \geq \frac{c}{2}t\quad \forall t \in [r,T).$$
    Next, we define the set 
    $$ \bar{\mathcal{P}} = \{ U \in \R^5\mid \exists \{U_n\} \subset \mathcal{P} \text{ with } U_n \to u\text{ and } \limsup_{n\to \infty} \eta(U_n) < \infty\}$$
    which includes the vacuum states. 
    On this set we extend $\eta$ such that it lower semicontinuous; 
    $$ \eta(U) = \liminf_{V \to U} \eta(V), $$
    where the $V$ in the liminf are taken to reside within $\mathcal{P}$. 
    It is known that this extension is convex (see Lemma 2 of \cite{VASSEUR2008323}).
    We have thus shown that for some $ C > 0 $ we have the bound
    $$ \min\left(|U-V|,|U-V|^2\right) \leq C(-\rho s)(U|V)\quad \forall V \in W,\ U \in \bar{\mathcal{P}},\ a \in [0,1].$$
    Next, noting that for all $x,y,z \in \R$ we have the simple inequality 
    $$\min\left( |x|^2, |x|\right) + \min\left(|y|^2, |y|\right) + \min\left(|z|^2, |z|\right) \leq 2 \min\left( |x|^2 + |y|^2 + |z|^2, \sqrt{|x|^2 + |y|^2+|z|^2}\right),$$
    that $\bar \theta=\bar \theta(V)$ is bounded above and below for all $V \in W$, and the definitions~\eqref{eq:G-coercive},~\eqref{eq:F-coercive} we arrive at 
    $$   C \left( \mathcal{F}(\rho) + \mathcal{G}(m) + \mathcal{G}(\mathcal{E} - \bar{\mathcal{E}}^a)\right) \leq \bar \theta (-\rho s)(U| V).$$
    The constant $C> 0$ is independent of the value of $V \in W$, so this completes the proof. 
\end{proof}

Now that we have established the form of the relative entropy, we use the shifted primitives $\bar \rho_{\delta,\epsilon},\bar\theta_{\delta,\epsilon}$ and compute the change in the relative entropy in time. 
To begin, we define 
\begin{equation}
    \label{eq:F-def}
    F_{\delta,\epsilon}(t,x) = \bar\theta_{\delta,\epsilon} (-\rho s) (U(x,t) |  V^a_{\delta,\epsilon}) 
\end{equation}
to be our relative entropy.
\begin{lemma} \label{lem:rel-entropy-shifts}
    Let $U = (\rho, m, \mathcal{E})$ be an admissible solution of the Navier--Stokes--Fourier system~\eqref{eq:nsf-system} with parameters $\kappa,\nu > 0$ and $a\in (0,1)$ written in conservative variables. Let $ V^a_{\delta,\epsilon}$ be the shifted profile of the entropy wave $V$ constructed in \Cref{sec:shift}.  
    Then, the weighted relative entropy satisfies
    \begin{equation} \label{eq:rel-entropy-inequality}
        \begin{aligned}
            \int_{\T^3} F_{\delta,\epsilon}(t,x) \,dx &+ \int_0^t \int_{\T^3}  \bar\theta_{\delta,\epsilon}\left( -\kappa \frac{\bf{q}(\theta,\nabla \theta)\cdot \nabla \theta}{\theta^2} + \nu \frac{\S(\theta,\nabla u):\nabla u}{\theta} \right)\,dx\,d\tau \\
            \leq&\,\kappa \int_0^t \int_{\T^3} K (\theta)\frac{\nabla \theta}{\theta}\cdot \nabla \bar \theta_{\delta,\epsilon}\,dx\,d\tau  \\
            &- \int_0^t \int_{\T^3} \rho s \left(\partial_t \bar \theta_{\delta,\epsilon} + u\cdot \nabla \bar \theta_{\delta,\epsilon}\right)\,dx\,d\tau\\
            &-\int_0^t \int_{\T^3} \rho \left(\partial_t \bar c_{\delta,\epsilon} + u\cdot \nabla \bar c_{\delta,\epsilon}\right)\,dx\,d\tau\\
            &+ \int_{\T^3} F_{\delta,\epsilon}(0,x)\,dx + Ca,
        \end{aligned}
    \end{equation}
    where $\bar \theta_{\delta,\epsilon}(t,x)$ is the temperature at state $V_{\delta,\epsilon}(t,x)$ and $\bar c_{\delta,\epsilon}(t,x) = \bar e_{\delta,\epsilon} - \bar \theta_{\delta,\epsilon} \bar s_{\delta,\epsilon} + \bar p_{\delta,\epsilon}/\bar \rho_{\delta,\epsilon}$, with $\bar p_{\delta,\epsilon} = \bar p_M + \frac{a}{3} \bar\theta^4_{\delta,\epsilon}$. 
    Again, all thermodynamic quantities correspond to the state $ V^a_{\delta,\epsilon}(t,x)$.
\end{lemma}
\begin{proof}
    We note that by Lemma \ref{lem:rel-entropy-eval}, the molecular pressure associated to the states of $V_{\delta,\epsilon}$ having constant value $\overline p_M$, and conservation of energy~\eqref{eq:E-con} that 
    \begin{equation} \label{eq:lemma-ent-inq-1}
    \begin{aligned}
        \left.\int_{\T^3} F_{\delta,\epsilon}(\cdot,x) \,dx\right|_0^t 
        &= \left.\int_{\T^3} \mathcal{E}(t,x)+ \overline p_{\delta,\epsilon}(\cdot,x) -\overline \theta_{\delta,\epsilon}(\cdot,x) (\rho s)(\cdot,x) - \rho(\cdot,x) \overline c_{\delta,\epsilon}(\cdot,x) \,dx\right|_0^t \\
        &= -\left.\int_{\T^3} \overline \theta_{\delta,\epsilon}(\cdot,x) (\rho s)(\cdot,x)\,dx \right|_0^t -   \left.\int \rho(\cdot,x) \overline c_{\delta,\epsilon}(\cdot,x) \,dx\right|_0^t \\
        &\qquad+ \frac{a}{3}\left.\int_{\T^3} \bar\theta_{\delta,\epsilon}(\cdot,x)^4\,dx\right|_0^t.
    \end{aligned}
    \end{equation}
    for almost every $t$. 
    Since $(\rho, m, \mathcal{E})$ is an admissible solution to the Navier--Stokes--Fourier system we wish to use $\bar c_{\delta,\epsilon},\bar \theta_{\delta,\epsilon}$ as test functions in equations~\eqref{eq:nsf-cont} and~\eqref{eq:nsf-entropy-inequality}. We note that as defined, $\bar c_{\delta,\epsilon},\bar \theta_{\delta,\epsilon}$ are not in $C^1([0,T]\times \T^3)$, but we have that $\bar c_{\delta,\epsilon},\bar \theta_{\delta,\epsilon}\in L^2(0,T; H^k(\T^3))$ and $\partial_t\bar c_{\delta,\epsilon}, \partial_t\bar \theta_{\delta,\epsilon} \in L^2(0,T; L^4(\T^3))$ for all $k \geq 0$ by the $\epsilon$ mollification.
    Since we have $(-\rho s),\rho \in L^\infty([0,T]; L^{4/3} (\T^3))$ thanks to the estimates~\eqref{eq:rho-integability}, \Cref{lemma:rho-s-integrability}, and $\gamma >2$ we can establish equations~\eqref{eq:nsf-cont} and~\eqref{eq:nsf-entropy-inequality} with $\bar c_{\delta,\epsilon},\bar \theta_{\delta,\epsilon}$ as test functions via an approximation argument. 
    For the first term, 
    \begin{align*}
        \left.-\int_{\T^3} \overline \theta_{\delta,\epsilon}(\cdot,x) (\rho s)(\cdot,x)\,dx \right|_0^t &\leq -\int_0^t\int_{\T^3} \rho s\left( \partial_t \bar \theta_{\delta,\epsilon} + u\cdot \nabla \bar \theta_{\delta,\epsilon}\right) + \kappa \frac{\mathbf{q}(\theta,\nabla \theta)\cdot \nabla \bar \theta_{\delta,\epsilon}}{\theta} \,dx\,d\tau \\
        &\hphantom{\leq}-\int_0^t \int_{\T^3} \frac{\bar \theta_{\delta,\epsilon}}{\theta} \left( {\nu}\ \S(\theta,\nabla u):\nabla u - \kappa \frac{\mathbf{q}(\theta,\nabla \theta)\cdot \nabla \theta}{\theta} \right)\,dx\,d\tau.
    \end{align*}
    For the second, we find
    \begin{align*}
        \left.\int \rho(\cdot,x) \overline c_{\delta,\epsilon}(\cdot,x) \,dx\right|_0^t &= \int_0^t\int_{\T^3} \rho\left( \partial_t \bar c_{\delta,\epsilon} + u\cdot \nabla \bar c_{ \delta,\epsilon}\right)\,dx\,d\tau.
    \end{align*}
    The final term of~\eqref{eq:lemma-ent-inq-1} can be easily bounded thanks to $\T^3$ being a finite measure space and the bound~\eqref{eq:psi-bounded}.
    Substituting these into~\eqref{eq:lemma-ent-inq-1} and recalling~\eqref{eq:heat-flux} we arrive at the bound~\eqref{eq:rel-entropy-inequality}.
\end{proof}

We now simplify the bound attained in \Cref{lem:rel-entropy-shifts}.
Combining~\eqref{eq:visc-assum} and~\eqref{eq:heat-cond-assum} with \Cref{lemma:gen-korn} we find the bounds
\begin{align*}
\kappa \int_0^t\int_{\T^3} -\frac{\mathbf{q}(\theta,\nabla \theta)\cdot \nabla \theta}{\theta^2}\,dx\,dt &\geq  C\kappa\int_0^t \int_{\T^3}  \frac{1+\theta^3}{\theta^2} | \nabla \theta|^2\,dx\,dt,
\end{align*}
while the viscous dissipation term is non-negative.

Next, since $\mathcal{G}$ is Lipschitz, using our constitutive relation for the radiative energy~\eqref{eq:energy}, we have the bound 
$$ \mathcal{G}(\mathcal{E} - \bar {\mathcal{E}}) \leq \mathcal{G}(\mathcal{E} - \bar {\mathcal{E}}^a) + 2a \bar \theta_{\delta,\epsilon}^4.$$

From this, \Cref{lemma:coercive}, and $\bar \theta_{\delta,\epsilon}$ being bounded above there exists $C > 0$ such that for all $T > 0$ the following bound holds,
\begin{equation} \label{eq:rel-entropy-simp}
    \begin{aligned}
        \esssup_{t \in [0,T]} \left[ \int_{\T^3}  \mathcal{G}(m) + \mathcal{F}(\rho) + \mathcal{G}(\mathcal{E} - \bar{\mathcal{E}})\,dx \right]&+ \int_0^T \int_{\T^3}  \kappa \frac{1+\theta^3}{\theta^2}|\nabla \theta|^2  \,dx\,d\tau \\
        \leq&\, C \kappa \int_0^T\int_{\T^3} \frac{1+\theta^3}{\theta} |\nabla \theta | |\nabla \bar\theta_{\delta,\epsilon}|\,dx\,d\tau  \\
        &+ C\esssup_{t \in [0,T]}\left|\int_0^t \int_{\T^3} \rho [\partial_t \bar c_{\delta,\epsilon} + u\cdot \nabla \bar c_{\delta,\epsilon}] \,dx\,d\tau \right|\\   
        &+ C \esssup_{t \in [0,T]}\left|\int_0^t \int_{\T^3} \rho s [\partial_t \bar \theta_{\delta,\epsilon} + u\cdot \nabla \bar \theta_{\delta,\epsilon}] \,dx\,d\tau\right|\\
        &+ C\int_{\T^3} F_{\delta,\epsilon}(0,x)\,dx + Ca. \\
    \end{aligned}
\end{equation}

\subsection{Vanishing dissipation limit} \label{sec:diss-limit}
Since the coercive quantities are convex we can use weak lower semi-continuity to gain bounds on the weak limits of $\rho$, $m$, and $\mathcal{E}$ when the dissipation parameters, $a$, and the mollification parameters tend to zero. 
We choose to take our limits in the order $\kappa \to 0,\ a \to 0,\ \delta\to 0,\ \epsilon\to 0,$ and finally $\nu \to 0$. 
This order can be quickly justified by examining the terms appearing on the right hand side of the inequality~\eqref{eq:rel-entropy-simp},
\begin{itemize}
    \item The limit $\kappa \to 0$ is performed prior to $a \to 0$ to maintain higher integrability of $\theta$. Furthermore, it is completed prior to $\delta, \epsilon \to 0$ to take advantage of having a uniform estimate on $|\nabla \bar \theta_{\delta,\epsilon}|$ from the mollification.
    \item We perform the limit $a\to 0$ prior to $\delta \to 0$ and $\epsilon \to 0$ because the radiative entropy term $\rho s_R = \frac{4}{3} \theta^3$ is merely $L^{4/3}$, which causes difficulty as the transport terms do not converge to zero weakly in $L^p$ for $p > 2$ when $\delta,\epsilon \to 0$. 
        In the limit $a\to 0$ we observe that $\rho s \to \rho s_M$, which allows us to conclude that our weak limit has a higher degree of integrability by \Cref{lemma:rho-s-integrability}.
    \item Finally, $\delta \to 0,\epsilon \to 0$ is performed prior to $\nu \to 0$ because we wish to use the uniform $L^2$ bounds of $\div u$ to help control the terms arising from our mollified transport equation.
\end{itemize}
Furthermore, in the preceding sections, our solutions depended on fixed constants $\kappa,\ a,\ \nu > 0$, where we neglected to denote the dependence of our solutions on these constants.
For clarity, in this section we further notate our quantities with superscripts $\kappa,\ a,\ \nu $ to keep track of this dependence. 
After each limit, we will drop the associated superscript.
Additionally, when mentioning the entropy, we will write $(\rho s)^{\kappa,a,\nu}$ rather than $\rho^{\kappa, a, \nu} s^{\kappa, a, \nu}$. This is done because there are no a priori estimates allowing us to gain strong limits of either $\rho$ or $s$, so we instead obtain weak limits such as $(\rho s)^{\kappa,a,\nu} \wto (\rho s)^{a,\nu}$ where $(\rho s)^{a,\nu}$ may not be equal to the product of weak limits $\rho^{a, \nu} s^{ a, \nu}$.

\subsubsection{Heat conductivity limit ($\kappa \to 0$)} 
{The goal of this subsection is to perform the heat conductivity limit, $\kappa \to 0$, in~\eqref{eq:rel-entropy-simp}.
This is accomplished by studying \eqref{eq:rel-entropy-simp} term by term, using the a priori estimates obtained in \Cref{sec:prelim-bounds}, the Aubin--Lions lemma, and results about commutators.}

We begin by noting that Young's inequality gives the bound
\begin{equation}
    \begin{aligned}
    C\kappa \int_0^T \int_{\T^3} \frac{1+(\theta^{\kappa,a,\nu})^3}{\theta^{\kappa,a,\nu}} |\nabla \theta^{\kappa,a,\nu} | |\nabla \bar \theta^{\kappa,a,\nu}_{\delta,\epsilon}|\,dx\,d\tau 
        \leq&\,\frac{\kappa}{2} \int_0^T \int_{\T^3} \frac{1+(\theta^{\kappa,a,\nu})^3}{(\theta^{\kappa,a,\nu})^2} |\nabla \theta^{\kappa,a,\nu} |^2\,dx\,d\tau \\
        &+ C\kappa \int_0^T \int_{\T^3} (1+(\theta^{\kappa,a,\nu})^3) |\nabla \bar \theta_{\delta,\epsilon}^{\kappa,a,\nu}|^2\,dx\,d\tau.
    \end{aligned}
\end{equation}
This allows us to rewrite~\eqref{eq:rel-entropy-simp} as
\begin{multline} \label{eq:rel-entropy-simp-J1}
        \esssup_{t\in[0,T]} \left[ \int_{\T^3} \mathcal{G}(m^{\kappa,a,\nu}) + \mathcal{F}(\rho^{\kappa,a,\nu}) + \mathcal{G}(\mathcal{E}^{\kappa,a,\nu} - \bar{\mathcal{E}} )\,dx \right] + \int_0^T \int_{\T^3} \kappa \frac{1+(\theta^{\kappa,a,\nu})^3}{2(\theta^{\kappa,a,\nu})^2}|\nabla \theta^{\kappa,a,\nu}|^2 \,dx\,d\tau \\
        \leq\, C\left(\int_{\T^3} F_\epsilon^a(0,x)\,dx + Ca + J_1 + \sup_{t \in [0,T]} |J_2| + \sup_{t \in [0,T]} |J_3| \right),
\end{multline}
where 
\begin{equation} \label{eq:J-def}
    \begin{aligned}
    J_1&= \kappa \int_0^T \int_{\T^3} (1+(\theta^{\kappa,a,\nu})^3) |\nabla \bar \theta_{\delta,\epsilon}^{\kappa,a,\nu}|^2\,dx\,d\tau, \\
    J_2 &= \int_0^t \int_{\T^3} \rho^{\kappa,a,\nu}h'(\bar\theta_{\delta,\epsilon}^{\kappa,a,\nu}) \left[\partial_t \bar\theta_{\delta,\epsilon}^{\kappa,a,\nu} + u^{\kappa,a,\nu}\cdot \nabla  \bar\theta_{\delta,\epsilon}^{\kappa,a,\nu}\right] \,dx\,d\tau, \\
    J_3 &= \int_0^t \int_{\T^3}(\rho s)^{\kappa, a, \nu}  \left[\partial_t \bar\theta_{\delta,\epsilon}^{\kappa,a,\nu} + u^{\kappa,a,\nu}\cdot \nabla  \bar\theta_{\delta,\epsilon}^{\kappa,a,\nu}\right] \,dx\,d\tau,
    \end{aligned}
\end{equation}
and $h$ is the $C^1$ function satisfying $h(\bar\theta_{\delta,\epsilon}^{\kappa,a,\nu}) = \bar c_{\delta,\epsilon}^{\kappa,a,\nu}$.
Note that since $\bar \theta^{\kappa, a,\nu}_{\delta, \epsilon} = \bar\theta^{\kappa, a,\nu}_{\delta} \ast \varphi_\epsilon$ and $\bar\theta^{\kappa, a,\nu}_{\delta}$ is bounded~\eqref{eq:psi-bounded} we have the bound 
$$ \int_0^T \int_{\T^3} (1+(\theta^{\kappa,a,\nu})^3) |\nabla \bar\theta_{\delta ,\epsilon}^{\kappa,a,\nu}|^2\,dx\,d\tau \leq C(\epsilon,T)\left(1 +  \norm{\theta^{\kappa,a,\nu}}_{L^\infty(0,T;L^4(\T^3))}^3\right). $$ 
Additionally, by conservation of energy~\eqref{eq:bounds-from-energy} we have a uniform in $\kappa$ bound on $\norm{\theta^{\kappa,a,\nu}}_{L^\infty(0,T;L^4(\T^3))}$.
Therefore, 
$$  J_1 \to 0\quad \text{as}\quad \kappa \to 0.$$


To begin examining $J_2$ and $J_3$, we define the transport defect, which, as we will see later, corresponds to a commutator,
\begin{align}
    R^{\kappa,a,\nu}_{\delta,\epsilon} &:=\partial_t \bar\theta_{\delta,\epsilon}^{\kappa,a,\nu} + u^{\kappa,a,\nu}\cdot \nabla  \bar\theta_{\delta,\epsilon}^{\kappa,a,\nu},
\end{align}
giving $J_2$ and $J_3$ the form
\begin{align*}
    J_2 &=  \int_0^t \int_{\T^3} \rho^{\kappa,a,\nu}  h'(\bar\theta_{\delta,\epsilon}^{\kappa,a,\nu})R^{\kappa,a,\nu}_{\delta,\epsilon}\,dx\,d\tau,\\
    J_3 &=  \int_0^t \int_{\T^3} (\rho s)^{\kappa,a,\nu} R^{\kappa,a,\nu}_{\delta,\epsilon}\,dx\,d\tau. 
\end{align*}

To take the limits of $J_2$ and $J_3$ we will show 
\begin{align*}
    \rho^{\kappa,a,\nu} \to \rho^{a,\nu},\quad (\rho s)^{\kappa,a,\nu} \to (\rho s)^{a,\nu} \quad \quad\quad \text{in}&\quad L^2(0,T;H^{-1}(\T^3)) \quad\text{as}\quad\kappa\to0,\\
     R^{\kappa,a,\nu}_{\delta,\epsilon} \wto  R^{a,\nu}_{\delta,\epsilon}\quad \text{weakly in}&\quad L^2(0,T;H^{1}(\T^3))\quad\text{as}\quad\kappa\to 0,
\end{align*} 
and $\bar\theta_{\delta,\epsilon}^{\kappa,a,\nu} \to \bar \theta_{\delta,\epsilon}^{a,\nu}$ in $L^\infty([0,T]\times \T^3)$. 
To obtain the strong convergence of $\rho^{\kappa,a,\nu}$ and $(\rho s)^{\kappa, a,\nu}$ we will make use of the following variant of the Aubin--Lions Lemma.
\begin{lemma}[{\cite[Lemma 6.3]{feireisl04}}] \label{lem:aubin-lions}
    Let $\{v_n\}_{n=1}^\infty$ be a sequence of functions such that the $v_n$ are uniformly bounded in $L^2(0,T; L^q(\T^3)) \cap L^\infty (0,T; L^1(\T^3))$ where $q > \frac{6}{5}$. 
    Additionally, suppose 
    $$ \partial_t v_n \geq g_n \quad \text{in}\quad \mathcal{D}'( (0,T) \times \T^3)$$
    with 
    \beq \label{eq:g-assum} 
        g_n\quad \text{uniformly bounded in}\quad  L^1(0,T; W^{-m,r}(\T^3))\quad \text{for} \quad m\geq 1,\ r>1.\eeq
    Then there exists a subsequence of $\{v_n\}_{n=1}^\infty$ such that 
    $$ v_n \to v \quad \text{in}\quad L^2(0,T; H^{-1}(\T^3)). $$
\end{lemma}

To begin, we note that by Lemma \ref{lemma:rho-s-integrability} $(\rho s_M)^{\kappa,a,\nu}$ is bounded in $L^2(0,T; L^2(\T^3)) \cap L^\infty (0,T; L^1(\T^3))$. 
For the radiative component we note that by conservation of energy~\eqref{eq:bounds-from-energy}, $\norm{\theta^{\kappa,a,\nu}}_{L^\infty(0,T;L^4(\T^3))}$ is uniformly bounded with respect to $\kappa$. 
Hence, by the definition of $s_R$~\eqref{eq:entropy-def}, we have 
\begin{equation} \label{eq:radiative-entropy-bound-kappa-limit} 
    a(\rho s_R)^{\kappa,a,\nu} \quad\text{is bounded in}\quad L^\infty(0,T; L^{4/3}(\T^3)),
\end{equation} 
and thanks to $L^p$ nesting we have
\begin{equation}\label{eq:rho-s-bound-KAPPA} 
    (\rho s)^{\kappa,a,\nu} \quad \text{is bounded in}\quad L^2(0,T; L^{4/3}(\T^3)) \cap L^\infty (0,T; L^1(\T^3)),
\end{equation}
where we note that $q = \frac{4}{3} > \frac{6}{5}$.
Next, by the entropy inequality~\eqref{eq:nsf-entropy-inequality} we know that $(\rho s)^{\kappa,a,\nu}$ distributionally satisfies
\beq \label{eq:entropy-diff-ineq-AL}\begin{aligned}
    \partial_t (\rho s)^{\kappa,a,\nu} \geq& -\div( (\rho s)^{\kappa,a,\nu} u^{\kappa,a,\nu})- \kappa\div \left(\frac{\mathbf{q}(\theta^{\kappa,a,\nu}, \nabla \theta^{\kappa,a,\nu})}{\theta^{\kappa,a,\nu}}\right) \\
    &+ \frac{1}{\theta^{\kappa,a,\nu}} \left( \nu\ \S(\theta^{\kappa,a,\nu}, \nabla u^{\kappa,a,\nu}):\nabla u^{\kappa,a,\nu} - \kappa\frac{\mathbf{q}(\theta^{\kappa,a,\nu}, \nabla \theta^{\kappa,a,\nu})}{\theta^{\kappa,a,\nu}} \cdot \nabla \theta^{\kappa,a,\nu} \right) \\
    =:& \sum_{i=1}^4 g_i^\kappa.
\end{aligned} \eeq

For $g_1^\kappa$ we note that
$$ \norm{(\rho s)^{\kappa,a,\nu} u^{\kappa,a,\nu}}_{L^2(0,T; L^{12/11}(\T^3))} \leq \norm{(\rho s)^{\kappa,a,\nu}}_{L^\infty (0,T; L^{4/3}(\T^3))}  \norm{u^{\kappa,a,\nu}}_{L^2(0,T; L^6(\T^3))}$$ 
yielding an upper bound in $L^2(0,T; W^{-1,12/11}(\T^3))$ uniform in $\kappa$ and $a$ due to \Cref{lemma:rho-s-integrability}, the bound~\eqref{eq:u-H1-bound} and equation~\eqref{eq:radiative-entropy-bound-kappa-limit}. 

Next, {thanks again to the H\"older inequality}, we observe {that} for $g_2^\kappa$
\beq \label{eq:q-kappa-bound} 
\begin{aligned}
    \norm{\kappa \frac{\mathbf{q}(\theta^{\kappa,a,\nu}, \nabla \theta^{\kappa,a,\nu})}{\theta^{\kappa,a,\nu}}}_{L^2(0,T;L^{8/7}(\T^3))} &\leq \kappa\norm{\sqrt{K(\theta^{\kappa,a,\nu})}}_{L^\infty(0,T; L^{8/3}(\T^3))} \\
    &\qquad\times\norm{\frac{\sqrt{K(\theta^{\kappa,a,\nu})}}{\theta^{\kappa,a,\nu}}\nabla \theta^{\kappa, a,\nu}}_{L^2(0,T; L^2(\T^3))} \\
    &\leq  C\kappa \left( 1 + \norm{\theta}_{L^\infty(0,T; L^4(\T^3))}^{3/2}\right) \\
    &\qquad\times\norm{\frac{\sqrt{K(\theta^{\kappa,a,\nu})}}{\theta^{\kappa,a,\nu}}\nabla \theta^{\kappa, a,\nu}}_{L^2(0,T; L^2(\T^3))},
\end{aligned}\eeq
where the right is uniformly bounded for $\kappa \in (0,1)$ and $a$ fixed due to~\eqref{eq:bounds-from-energy} and~\eqref{eq:bounds-from-diss} giving a uniform bound on $g_2^\kappa \in L^2(0,T; W^{-1,\, 8/7}(\T^3))$.
We note that both $g_3,g_4 \geq 0$. 
This establishes by \Cref{lem:aubin-lions} that
$$ (\rho s)^{\kappa, a,\nu} \to (\rho s)^{a,\nu} \quad \text{in} \quad L^2(0,T; H^{-1}(\T^3))\quad \text{as}\quad \kappa \to 0.$$

For the density, we recall that $\rho^{\kappa,a,\nu}$ is uniformly bounded in $ L^\infty(0,T; L^\gamma(\T^3)) $ by energy conservation and our constitutive relation~\eqref{eq:rho-integability}.
Additionally, the continuity equation gives us that
$$ \partial_t \rho^{\kappa,a,\nu} = - \div(u^{\kappa,a,\nu} \rho^{\kappa,a,\nu}) $$
from which, thanks to Hölder's inequality,
$$ \norm{u^{\kappa,a,\nu} \rho^{\kappa,a,\nu}}_{L^1(0,T; L^q(\T^3))} \leq \norm{u^{\kappa,a,\nu}}_{L^2(0,T; L^6(\T^3))} \norm{\rho^{\kappa,a,\nu}}_{L^2(0,T; L^\gamma(\T^3))},$$ 
with $q=\frac{6\gamma}{\gamma+6}>1$ since $\gamma>2$. Using equations~\eqref{eq:u-H1-bound} and~\eqref{eq:rho-integability}, the right hand side is uniformly bounded with respect to $\kappa$ and $a$.
This shows, {again using \Cref{lem:aubin-lions}, that} 
$$ \rho^{\kappa,a,\nu} \to \rho^{a,\nu} \quad \text{in}\quad L^2(0,T;H^{-1}(\T^3))\quad \text{as}\quad \kappa \to 0$$ 
as desired.

For the transport term $ R_{\delta,\epsilon}^{\kappa,a,\nu} $ we first note that, by mollifying the transport equation~\eqref{eq:delta-shift-pde} by $\varphi_\epsilon$ as defined in~\eqref{eq:def-mol}, we have 
$$ \partial_t \bar\theta^{\kappa,a,\nu}_{\delta,\epsilon}  = -\left(u^{\kappa, a,\nu}_\delta \cdot \nabla \bar\theta^{\kappa,a,\nu}_\delta\right)\ast\varphi_\epsilon = -\div\left([u^{\kappa, a,\nu}_\delta \bar\theta^{\kappa,a,\nu}_\delta]_\epsilon\right) + [\div(u^{\kappa, a,\nu}_\delta )\bar\theta^{\kappa,a,\nu}_\delta]_\epsilon.$$
Using this we rewrite $ R^{\kappa,a,\nu}_{\delta,\epsilon} $ as
$$ R^{\kappa,a,\nu}_{\delta,\epsilon} = u^{\kappa,a,\nu} \cdot \nabla \bar\theta^{\kappa,a,\nu}_{\delta,\epsilon} - \div ([u^{\kappa,a,\nu}_\delta\bar\theta^{\kappa,a,\nu}_\delta]_\epsilon) +[\div(u^{\kappa,a,\nu}_\delta) \bar\theta^{\kappa,a,\nu}_\delta]_\epsilon. $$
Let us now show that 
\begin{equation}
\label{eq:weak_conv_R}
    R^{\kappa,a,\nu}_{\delta,\epsilon} \wto R^{a,\nu}_{\delta,\epsilon} \qquad \hbox{in} \quad L^2(0,T; H^1(\T^3)).
\end{equation}
Observe that $\bar\theta^{\kappa,a,\nu}_{\delta,\epsilon}$ is uniformly bounded in $L^\infty(0,T; H^2(\T^3))$ as $|\bar\theta^{\kappa, a, \nu}_\delta| \leq C$ by~\eqref{eq:psi-bounded} and the mollification in $\epsilon$. 
By definition, it further satisfies 
\begin{equation} \label{eq:kappa-lim-transport} \partial_t \bar\theta^{\kappa,a,\nu}_\delta = -u^{\kappa,a,\nu}_\delta \cdot \nabla \bar\theta^{\kappa,a,\nu}_\delta, \end{equation}
where equation~\eqref{eq:u-H1-bound} and mollification by $\delta$ give that the right side is  bounded in $L^2(0,T; H^{-1}(\T^3))$ uniformly with respect to the parameters $\kappa$ and $a$.
Combining this with the bound~\eqref{eq:psi-bounded} we can apply the Aubin--Lions lemma to deduce that
\begin{equation} \label{eq:kappa-psi-strong-cv} \bar\theta^{\kappa,a,\nu}_\delta \to \bar\theta^{a,\nu}_\delta \quad \text{in}\quad L^2(0,T; H^{-1}(\T^3))\quad \text{as}\quad\kappa \to 0.\end{equation}
Next, we note that mollification by $\delta$ gives us $\norm{u^{\kappa,\nu}_\delta}_{L^2(0,T; H^2(\T^3))} \leq C$ where $C$ is uniform with respect to $\kappa$ and $a$.

Now that we are in possession of the strong limit of $\bar\theta_{\delta,\epsilon}^{\kappa,a,\nu}$, we introduce the following lemma to obtain weak limits for terms in $R_{\delta,\epsilon}^{\kappa, a, \nu}$.
\begin{lemma}[Lemma 3.5 of \cite{kvy21}] \label{lem:kyv-R-cv}
    Assume that $\{u_n\}_{n=1}^\infty$ and $\{\psi_n\}_{n=1}^\infty$ satisfy one of the following:
    \begin{enumerate}
        \item $u_n \to u$ in $L^2(0,T; H^1(\T^3))$ and $\psi_n \wsto \psi$ in $L^\infty((0,T)) \times \T^3)$,
        \item $\psi_n \to \psi$ in $L^2(0,T; H^{-1}(\T^3))$ and 
        $$\norm{\psi_n}_{L^\infty( (0,T)\times \T^3)} + \norm{u_n}_{L^2(0,T; H^1(\T^3))} + \norm{ \nabla \div\, u_n}_{L^2(0,T; L^2(\T^3))} < C,$$ with $C$ independent of $n$. 
    \end{enumerate}
    There, for any $\epsilon > 0$ there exists a subsequence such that 
    \begin{align*}
        ( \div(u_n \psi_n))_\epsilon &\wto (\div(u\psi))_\epsilon \quad \text{weakly in}\quad L^2(0,T;H^1(\T^3)), \\
        ( \div(u_n) \psi_n)_\epsilon &\wto (\div(u)\psi)_\epsilon \quad \text{weakly in}\quad L^2(0,T;H^1(\T^3))
    \end{align*}
    where $u$ satisfies that $u_n\wto u$ in $L^2(0,T;H^1(\T^3))$. 
\end{lemma}
Using \Cref{lem:kyv-R-cv}, we find that
\begin{align*}
        ( \div(u^{\kappa,a,\nu}_\delta \bar\theta^{\kappa,a,\nu}_{\delta}))_\epsilon &\wto (\div(u^{a,\nu}_\delta\bar\theta^{a,\nu}_\delta))_\epsilon \quad \text{weakly in} \quad L^2(0,T;H^1(\T^3)) \quad\text{as}\quad\kappa\to0, \\
        ( \div(u^{\kappa,a,\nu}_\delta) \bar\theta^{\kappa,a,\nu}_{\delta})_\epsilon &\wto (\div(u^{a,\nu}_\delta)\bar\theta^{a,\nu}_\delta)_\epsilon \quad
        \text{weakly in}\quad L^2(0,T;H^1(\T^3)) \quad\text{as}\quad\kappa\to0.
\end{align*}
This proves the weak convergence for the two rightmost terms of $R^{\kappa, a,\nu}_{\delta,\epsilon}.$
For the remaining term, we note that, by mollification by $\epsilon$ and the bound~\eqref{eq:u-H1-bound}, we have
$$ \norm{ u^{\kappa, a,\nu} \nabla \bar\theta^{\kappa, a,\nu}_{\delta,\epsilon}}_{L^2(0,T; H^1(\T^3))} \leq \norm{u^{\kappa, a, \nu}}_{L^2(0,T;H^1(\T^3))} \norm{\nabla \bar\theta^{\kappa, a, \nu}_{\delta,\epsilon}}_{L^\infty(0,T; W^{1,\infty}(\T^3))} $$
which is uniform with respect to $\kappa$ and $a$. 
This establishes that the product $u^{\kappa, a,\nu} \nabla \bar\theta^{\kappa, a,\nu}_{\delta,\epsilon}$ has a weak limit in $L^2(0,T; H^1(\T^3))$.
We can then conclude that this weak limit is equal to the product of weak limits, since due to convergence~\eqref{eq:kappa-psi-strong-cv} and mollification  
$$\nabla \bar\theta^{\kappa, a,\nu}_{\delta,\epsilon} \to \nabla\bar\theta^{a,\nu}_{\delta,\epsilon}\quad\text{in}\quad L^2(0,T;H^{-1}(\T^3))\quad \text{as}\quad \kappa\to 0.$$
This concludes the proof of \eqref{eq:weak_conv_R}.

Finally, since $\bar\theta_{\delta,\epsilon}^{\kappa,a,\nu}$ is uniformly bounded in $L^\infty(0,T; W^{2,\infty}(\T^3))$ and satisfies the mollified version of~\eqref{eq:kappa-lim-transport}, we note that, by the classical Aubin--Lions--Simon Lemma and the smoothness of $h$, 
$$ h'(\bar\theta_{\delta,\epsilon}^{\kappa,a,\nu}) \to h'(\bar\theta_{\delta,\epsilon}^{a,\nu})\quad\text{in}\quad C( 0,T;W^{1,\infty}(\T^3))\quad \text{as}\quad \kappa\to 0. $$
This establishes the desired weak limit
$$h'(\bar\theta_{\delta,\epsilon}^{\kappa,a,\nu})R^{\kappa,a,\nu}_{\delta,\epsilon} \wto  h'(\bar\theta_{\delta,\epsilon}^{a,\nu})R^{a,\nu}_{\delta,\epsilon}\quad \text{in}\quad L^2(0,T;H^{1}(\T^3))\quad \text{as}\quad \kappa\to 0, $$
which allows us to take the limit of $J_2$ as $\kappa \to 0$.

To complete our study of $(\rho s)^{a,\nu}$, we note that the norm $\norm{(\rho s_M)^{\kappa,a,\nu}}_{L^\infty( 0,T; L^2( \T^3))}$ is uniformly bounded with respect to $\kappa$ and $a$, by the conservation of energy~\eqref{eq:bounds-from-energy} and \Cref{lemma:rho-s-integrability}. 
Furthermore, by~\eqref{eq:entropy-def} the radiative entropy satisfies the bound
\begin{equation}
    \label{eq:kappa-lim-rad-entropy-control}
    \begin{aligned}
        \norm{(\rho s_R)^{\kappa,a,\nu}}_{L^\infty(0,T; L^{4/3}(\T^3 ))} &= \frac{4}{3}\norm{ (\theta^{\kappa,a,\nu})^3}_{L^\infty(0,T; L^{4/3}(\T^3))} \\
        &= \frac{4}{3} \norm{(\theta^{\kappa,a,\nu})^4}_{\L^\infty(0,T; L^1(\T^3))}^{3/4} \\
        &\leq \frac{4}{3} a^{-3/4} \norm{\mathcal{E}_0}_{L^1(\T^3)}^{3/4}
    \end{aligned}
\end{equation}
where the final line is due to~\eqref{eq:bounds-from-energy}.
As a result, up to a subsequence we have the limits 
\begin{equation}\label{eq:G-wlims-kappa}\begin{aligned}
    (\rho s_M)^{\kappa,a,\nu} &\wto (\rho s_M)^{a,\nu} &&\text{weakly-$\ast$ in}\quad L^\infty( 0,T; L^2(\T^3))\quad \ \ \text{as}\quad\kappa\to0,\\
    (\rho s_R)^{\kappa,a,\nu} &\wto (\rho s_R)^{a,\nu}&&\text{weakly-$\ast$ in}\quad L^\infty(0,T; L^{4/3}(\T^3 ))\quad \text{as}\quad \kappa\to0.
\end{aligned} \end{equation}
Then, we can identify the weak limit of the entropy as the sum of the weak limits of the corresponding molecular entropy and radiative entropy, 
$$ (\rho s )^{a,\nu} = (\rho s_M)^{a,\nu} + a(\rho s_R)^{a,\nu}.$$

To summarize this section, by the bounds~\eqref{eq:E-con} and~\eqref{eq:bounds-from-energy} we have the weak limits 
\begin{align*}
    \rho^{\kappa,a,\nu} &\wto \rho^{a,\nu}\quad \text{weakly-$\ast$ in}\quad L^\infty(0,T; L^\gamma(\T^3))\quad \text{as}\quad\kappa\to0, \\
    \mathcal{E}^{\kappa,a,\nu}&\wto \mathcal{E}^{a,\nu} \quad\text{weakly-$\ast$ in} \quad L^\infty(0,T; \mathcal{M}(\T^3))\quad \text{as}\quad\kappa\to0.
\end{align*}
Furthermore, we note that from the bound~\eqref{eq:bounds-from-energy} we control the kinetic energy, 
$$\norm{ \frac{m^{\kappa,a,\nu}}{2\sqrt{\rho^{\kappa,a,\nu}}}}_{L^\infty(0,T; L^2(\T^3))}^2 \leq \norm{\mathcal{E}_0}_{L^1(\T^3)}.$$
Combining this with our uniform bound of $\rho^{\kappa,a,\nu} \in L^\infty(0,T; L^\gamma(\T^3))$ from~\eqref{eq:rho-integability} yields 
\begin{equation} \label{eq:m-uniform-bound}\begin{aligned}
    \norm{m^{\kappa,a,\nu}}_{L^\infty(0,T; L^{2\gamma/(\gamma+1)}(\T^3; \R^3))} &= \norm{\sqrt{\rho^{\kappa,a,\nu}} \frac{m^{\kappa,a,\nu}}{\sqrt{\rho^{\kappa,a,\nu}}}}_{L^\infty(0,T; L^{2\gamma/(\gamma+1)}(\T^3; \R^3))} \\
    &\leq \norm{\rho^{\kappa,a,\nu}}_{L^\infty(0,T; L^\gamma(\T^3))}^{1/2} \norm{\frac{m^{\kappa,a,\nu}}{\sqrt{\rho^{\kappa,a,\nu}}}}_{L^\infty(0,T; L^2(\T^3;\R^3))}.
    \end{aligned}
\end{equation}
We note that the above bound only depends on $\mathcal{E}_0$ and is uniform in all parameters $\kappa$, $a$, and $\nu$. 
This uniform bound allows us to conclude that 
$$m^{\kappa,a,\nu}\wto m^{a,\nu} \quad\text{weakly-$\ast$ in} \quad L^\infty(0,T; L^{\frac{2\gamma}{\gamma + 1}}(\T^3; \R^3))\quad \text{as}\quad\kappa\to0$$
Taking the limit as $\kappa \to 0$ of the bound~\eqref{eq:rel-entropy-simp-J1} we see that the weak limits $\rho^{a,\nu}$ and $m^{a,\nu}$ satisfy
\begin{equation}
    \begin{aligned}
        \esssup_{t \in [0,T]} \int_{\T^3} \mathcal{G}(m^{a,\nu}(t,x)) + \mathcal{F}(\rho^{a,\nu}(t,x)) \,dx &\leq  C\esssup_{t \in [0,T]} \left|\int_0^t \int_{\T^3} \rho^{a,\nu} h'(\bar\theta_{\delta,\epsilon}^{a,\nu}) R^{a,\nu}_{\delta,\epsilon} \,dx\,d\tau\right| \\
        &\quad+C\esssup_{t \in [0,T]} \left|\int_0^t\int_{\T^3}(\rho s)^{a,\nu} R^{a,\nu}_{\delta,\epsilon} \,dx\,d\tau\right|\\
        &\quad+C\int F_\epsilon^a(0,x)\,dx + Ca.
    \end{aligned}
\end{equation}
Since the energy converges in $L^\infty(0,T; \mathcal{M}(\T^3))$ rather than in an $L^p$ space we handle it differently.
We note that $\mathcal{G}$ bounds the $L^1$ distance like so:
\begin{equation} \label{eq:L1-G-control}
    \int_{\T^3}|f|\,dx = \int_{|f| < 1} |f|\,dx + \int_{|f|\geq 1}|f|\,dx \leq \sqrt{\int \mathcal{G}(f)\,dx} + \int \mathcal{G}(f)\,dx.
\end{equation}
From this and~\eqref{eq:rel-entropy-simp-J1} we arrive at the following inequality, 
\begin{align} \label{eq:bound-after-kappa}
     \esssup_{t\in [0,T]} \int_{\T^3} \mathcal{G}(m^{a,\nu}) + \mathcal{F}(\rho^{a,\nu}) \,dx+\norm{\mathcal{E}^{a,\nu} - \bar {\mathcal{E}}}_{L^\infty(0,T; \mathcal{M}(\T^3))} 
        \leq& C\left(\sqrt{A^{a,\nu}_{\delta,\epsilon}} + A^{a,\nu}_{\delta,\epsilon} \right),
\end{align}
where 
\begin{align*} 
    A^{a,\nu}_{\delta,\epsilon} := \int_{\T^3} F_\epsilon^a(0,x)\,dx + a +&\esssup_{t\in[0,T]}\left|\int_0^t \int_{\T^3} \rho^{a,\nu} h'(\bar\theta_{\delta,\epsilon}^{a,\nu}) R^{a,\nu}_{\delta,\epsilon} \,dx\,d\tau\right|\\
    +&\esssup_{t\in[0,T]}\left|\int_0^t\int_{\T^3}(\rho s)^{a,\nu}  R^{a,\nu}_{\delta,\epsilon} \,dx\,d\tau\right|.
\end{align*}

Finally, we note that by taking the limit $\kappa \to 0$ of the continuity equation~\eqref{eq:nsf-cont} yields that the limits $\rho^{a,\nu}$ and $m^{a,\nu}$ satisfy the equation,  
\beq \label{eq:cont-post-k} 
    \partial_t \rho^{a,\nu} + \div(m^{a,\nu}) = 0 \quad \text{in} \quad \mathcal{D}'(\R_{\geq 0}\times \R^3). 
\eeq
Similarly, taking the limit of the entropy inequality~\eqref{eq:nsf-entropy-inequality} yields that $(\rho s)^{a,\nu}$ satisfies 
\beq \label{eq:ent-post-k} 
    \partial_t (\rho s)^{a,\nu} \geq \overline g_1 \quad \text{in} \quad \mathcal{D}'(\R_{\geq 0}\times \R^3), 
\eeq
where $\bar g_1$ is the weak limit $g_1 \wto \bar g_1$ as identified in the inequality~\eqref{eq:entropy-diff-ineq-AL}.
There is no contribution in the limit from the $g_2$ terms, as they converge to zero by~\eqref{eq:q-kappa-bound}.
In particular, since the $g_1$ are bounded uniformly in $a$, $\kappa$ in the space $ L^2(0,T; W^{-1,\,12/11}(\T^3))$
the weak limit $\bar g_1$ is bounded in the same space, uniformly in $a$.

\subsubsection{Radiative pressure limit {($a \to 0$)}} \label{sec:lim-in-a}
To take the radiative pressure term to zero we first note that $(\rho s)^{a,\nu}$ and $\rho^{a,\nu}$ both meet the same hypotheses as in the last section to apply Lemma \ref{lem:aubin-lions} (see equations~\eqref{eq:cont-post-k},~\eqref{eq:ent-post-k}, and the surrounding discussion).
This gives us the strong limits
\begin{align}
    \rho^{a,\nu}\to \rho^\nu,\quad (\rho s)^{a,\nu}\to (\rho s)^\nu \quad \text{in}\quad L^2(0,T; H^{-1}(\T^3)) \quad\text{as}\quad a \to 0.
\end{align}
Similarly, the bounds used to deduce the weak limit $ R^{\kappa, a,\nu}_\alpha \wto R^{a,\nu}_\alpha$ in $L^2(0,T; H^1(\T^3))$ in the previous section are uniform in $a$, so by the same argument as before we have
\begin{equation}
    R^{a,\nu}_{\delta,\epsilon}\wto R^{\nu}_{\delta,\epsilon} := u^\nu \cdot \nabla \bar\theta^\nu_{\delta,\epsilon} - \div([u_\delta^\nu \bar\theta^\nu_\delta]_{\epsilon}) + [\div(u^\nu_\delta)\bar\theta_\delta^\nu]_{\epsilon} \quad \text{weakly in}\quad L^2(0,T; H^{1}(\T^3)).
\end{equation}
and for the $J_2$ term we can still establish 
$$ h'(\bar\theta_{\delta,\epsilon}^{a,\nu}) R_{\delta,\epsilon}^{a,\nu} \wto h'(\bar\theta_{\delta,\epsilon}^{\nu})R_{\delta,\epsilon}^\nu \quad \text{weakly in} \quad L^2(0,T; H^1(\T^3)).$$

At this point, we again consider the structure of the strong limit $(\rho s)^\nu$. 
Recalling the bound~\eqref{eq:kappa-lim-rad-entropy-control},
$$\norm{ a(\rho s_R)^{a,\nu}}_{L^\infty(0,T; L^{4/3}(\T^3))} \leq \frac{4}{3}a^{1/4} \norm{\mathcal{E}_0}_{L^1(\T^3)}^{3/4},$$
by Sobolev embedding this bound controls the norm of $(\rho s_R)^{a,\nu}$ in the space in $L^2(0,T; H^{-1}(\T^3))$.
In particular, when $a \to 0$ we find the weak limit
$$ a(\rho s_R)^{a,\nu} \wto 0 \quad\text{in}\quad L^2(0,T; H^{-1}(\T^3))\quad\text{as}\quad a\to 0.$$
Moreover, since $\norm{(\rho s_M)^{a,\nu}}_{L^2( (0,T) \times \T^3)} $ is uniformly bounded independently of $a$, we obtain that our entropy has the form
$$ (\rho s)^\nu = (\rho s_M)^\nu.$$
In particular, thanks to the limits~\eqref{eq:G-wlims-kappa}, the conservation of energy~\eqref{eq:bounds-from-energy}, and Lemma \ref{lemma:rho-s-integrability} we have the $L^2$ bound
\beq \label{eq:Gnu-uni} \norm{(\rho s)^\nu }_{L^\infty( 0,T; L^2(\T^3))} < C, \eeq
for some $ C > 0$ depending on the initial data $(\rho_0, u_0, \theta_0)$ and $\gamma$.
This increase in integrability will be important later, as it will allow us to consider weak limits of $R_{\delta,\epsilon}^\nu$ in weaker topologies when we take our mollification parameters $\delta,\epsilon$ to zero.

This then gives that 
\begin{align*}
    J_2 &\to   \int_0^t \int_{\T^3} \rho^{\nu}h'(\bar\theta^{\nu}_{\delta,\epsilon}) \left[\partial_t \bar\theta_{\delta,\epsilon}^{\nu} + u^{\nu}\cdot \nabla  \bar\theta_{\delta,\epsilon}^{\nu}\right] \,dx\,d\tau\hphantom{(,)} \quad \text{as} \quad a\to 0,\\
    J_3 &\to  \int_0^t \int_{\T^3}(\rho s)^{\nu}   \left[\partial_t \bar\theta_{\delta,\epsilon}^{\nu} + u^{\nu}\cdot \nabla  \bar\theta_{\delta,\epsilon}^{\nu}\right] \,dx\,d\tau \quad \text{as} \quad a\to 0.
\end{align*}

To complete the limit of \eqref{eq:bound-after-kappa} we note that the term containing $F_\epsilon^a(0,\cdot )$ depends on $a$ due to~\eqref{eq:F-def} and \eqref{eq:entropy-def} containing a radiative component.
In this case, it is easy to see that $(-\rho s)^a(U_0) \to (-\rho s_M)(U_0), (-\rho s)^a(V^a_{\delta,\epsilon}) \to (-\rho s_M)(V_{\delta,\epsilon})$ in $L^1(\T^3)$. 
For the affine term, we use that $d(-\rho s)^a(V^a_{\delta,\epsilon}) = d(-\rho s_M)(V_{\delta,\epsilon})$ for all $a$; this follows from~\eqref{eq:gibbs-rho} and observing that the radiative components of $s, p, e$ all cancel for any state $V$. 
Finally, we note that in conservative coordinates (which we denote here as $U_0^a$; while $U_0$ is constant in $a$ when written in primitive variables here we emphasize the dependence of the energy on the parameter $a$) $$U_0^a - V^a = \left(\rho_0 - \bar \rho, \rho_0u_0, \rho_0 \frac{|u_0|^2}{2} +\rho_0 e_M(\rho_0,\theta_0) - \bar \rho e_M(\bar \rho, \bar \theta) \right) + a(0,0, \theta_0^4 - \bar \theta^4) \to U_0^0 - V$$
in $L^1(\T^3)$ as $a \to 0$, where we are setting $\bar \rho = \bar \rho_{\delta,\epsilon}^\nu(0,\cdot)$ and $\bar \theta = \bar \theta_{\delta,\epsilon}^\nu(0,\cdot)$. 
This establishes $F_\epsilon^a(0,\cdot ) \to F_\epsilon(0,\cdot )$ in $L^1(\T^3)$, giving convergence of these integrals. 

We also have weak limits of $\rho^{a,\nu}$, $m^{a,\nu}$, and $\mathcal{E}^{a,\nu}$ due to our a priori bounds identical to in the $\kappa \to 0$ limit, 
\begin{align*}
    \rho^{a,\nu} &\wto \rho^{\nu}\quad \text{weakly-$\ast$ in}\quad L^\infty(0,T; L^\gamma(\T^3))\quad \text{as}\quad a\to0, \\
    m^{a,\nu}&\wto m^{\nu} \quad\text{weakly-$\ast$ in} \quad L^\infty(0,T; L^{2\gamma/(\gamma + 1)}(\T^3;\R^3))\quad \text{as}\quad a\to0,\\
    \mathcal{E}^{a,\nu}&\wto \mathcal{E}^{\nu} \quad\text{weakly-$\ast$ in} \quad L^\infty(0,T; \mathcal{M}(\T^3))\quad \text{as}\quad a\to0.
\end{align*}
These limits enjoy the bound
\beq \label{eq:bound-after-a}\begin{aligned}
     \esssup_{t\in [0,T]} \int_{\T^3} G(m^\nu) + \mathcal{F}(\rho^{\nu}) \,dx +\norm{ \mathcal{E}^{\nu} - \bar{\mathcal{E}} }_{L^\infty(0,T; \mathcal{M}(\T^3))} 
        \leq& C \left(\sqrt{A^{\nu}_{\delta,\epsilon}} + A^{\nu}_{\delta,\epsilon}\right),
\end{aligned} \eeq
where
\begin{align*} 
    A^{\nu}_{\delta,\epsilon} := \int_{\T^3} F_\epsilon(0,x)\,dx &+  \esssup_{t\in[0,T]} \left|\int_0^t \int_{\T^3} \rho^{\nu} h'(\bar\theta_{\delta,\epsilon}^{\nu}) R_{\delta,\epsilon}^{\nu} \,dx\,d\tau \right| + \esssup_{t\in [0,T]} \left|\int_0^t \int_{\T^3} (-\rho s_M)^{\nu}  R_{\delta,\epsilon}^{\nu} \,dx\,d\tau \right|.
\end{align*}

\subsubsection{Mollification in $\delta$ limit ($\delta \to 0$)}
The only terms depending on $\delta$ within~\eqref{eq:bound-after-a} are $R^\nu_{\delta,\epsilon}$ and $h'(\bar \theta_{\delta,\epsilon}^\nu)$, so the next two sections will only concern these terms.
First, we recall
$$ R^\nu_{\delta,\epsilon} = u^\nu \cdot \nabla \bar \theta^\nu_{\delta,\epsilon} - \div([u_\delta^\nu \bar \theta^\nu_\delta]_{\epsilon}) + [\div(u^\nu_\delta)\bar \theta_\delta^\nu]_{\epsilon}.$$
Due to the uniform bound of $\bar \theta_{\delta}^\nu$ in $L^\infty ((0,T)\times \T^3)$ from equation~\eqref{eq:psi-bounded} there exists $\bar \theta^\nu\in L^\infty ((0,T)\times \T^3)$ such that 
\begin{equation*} \bar \theta_{\delta}^\nu \wto \bar \theta^\nu \quad  \text{weak-$\ast$ in}\quad  L^\infty((0,T)\times \T^3)\quad\text{as}\quad \delta \to 0.\end{equation*}
From this we also have the convergence of the mollified shifted temperature,
\begin{equation}\label{eq:delta-lim-psi-ws-Linf} \bar \theta_{\delta,\epsilon}^\nu \wto \bar \theta^\nu_{\epsilon} := \bar \theta^\nu \ast\varphi_\epsilon \quad  \text{weak$-\ast$ in}\quad  L^\infty((0,T)\times \T^3)\quad\text{as}\quad \delta \to 0.\end{equation}
Next, due to mollification in $\epsilon$, the $\bar \theta_{\delta,\epsilon}^\nu$ are uniformly bounded in $L^\infty(0,T; W^{1,\infty}(\T^3))$, which yields 
$$ u^\nu \nabla \bar \theta_{\delta,\epsilon}^\nu \wto u^\nu \nabla\bar \theta_\epsilon^\nu\quad \text{weakly in}\quad L^2(0,T; L^2(\T^3))\quad\text{as}\quad \delta\to0. $$
For the remaining two terms we note that since $u^\nu \in L^2(0,T; H^1(\T^3))$ by standard properties of mollifiers we have the strong convergence 
$$ u^\nu_\delta \to u^\nu \quad \text{in}\quad  L^2(0,T; H^1(\T^3))\quad \text{as} \quad \delta\to 0.$$
This combined with the weak convergence of $\bar \theta^\nu_{\delta,\epsilon}$~\eqref{eq:delta-lim-psi-ws-Linf} and \Cref{lem:kyv-R-cv} gives, up to a subsequence, the weak limits 
 \begin{align*}
    ( \div(u_\delta^\nu \bar \theta_\delta^\nu))_\epsilon &\wto (\div(u^\nu\bar \theta^\nu))_\epsilon \quad \text{in}\quad L^2(0,T;H^1(\T^3))\quad \text{as}\quad \delta\to 0, \\
    ( \div(u_\delta^\nu) \bar \theta_\delta^\nu)_\epsilon &\wto (\div(u^\nu)\bar \theta^\nu)_\epsilon \quad \text{in}\quad L^2(0,T;H^1(\T^3))\quad \text{as}\quad \delta\to 0.
\end{align*} 

We next note, via Aubin--Lions--Simon and the $\epsilon$ mollification, that we have the convergence 
$$\bar \theta_{\delta,\epsilon}^\nu \to \bar \theta_\epsilon^\nu \quad\text{in} \quad C([0,T]; W^{k,\infty}(\T^3)),$$
for any $k \in \N$. In particular, this gives the strong convergence
$$ h'(\bar\theta_{\delta,\epsilon}^\nu) \to h'(\bar\theta_{\epsilon}^\nu) \quad \text{in} \quad L^\infty([0,T]\times \T^3). $$

We have shown that 
$$ R^\nu_{\delta,\epsilon} \wto R_\epsilon^\nu\quad \text{in}\quad L^2(0,T;H^1(\T^3))\quad \text{as}\quad \delta\to 0, $$
where 
$$ R^\nu_{\epsilon} := u^\nu \cdot \nabla \bar \theta^\nu_{\epsilon} - \div([u^\nu \bar \theta^\nu]_{\epsilon}) + [\div(u^\nu)\bar \theta^\nu]_{\epsilon}.$$
By~\eqref{eq:rho-gamma} and~\eqref{eq:Gnu-uni} we have $\rho^\nu, (\rho s)^\nu \in L^2(0,T; L^2(\T^3))$ which gives us that 
\begin{align*}
    A^\nu_{\delta,\epsilon} \to  A^{\nu}_{\epsilon} :=& \int_{\T^3} F_\epsilon(0,x)\,dx + \esssup_{t\in [0,T]} \left|\int_0^t \int_{\T^3}  \rho^{\nu} h'(\bar \theta_{\epsilon}^{\nu}) R^{\nu}_{\epsilon} \,dx\,d\tau \right|\\&\qquad+\esssup_{t\in [0,T]}\left|\int_0^t \int_{\T^3}  (\rho s_M)^{\nu} R^{\nu}_{\epsilon} \,dx\,d\tau \right|\quad \text{as} \quad\delta\to 0.
\end{align*}
Therefore, passing to the limit $\delta \to 0$ in the bound~\eqref{eq:bound-after-a} we have 
\beq \label{eq:bound-after-delta}\begin{aligned}
     \esssup_{t \in [0,T]}\int_{\T^3} G(m^\nu) + \mathcal{F}(\rho^{\nu}) \,dx +\norm{\mathcal{E}^{\nu} -  \bar{\mathcal{E}} }_{L^\infty(0,T; \mathcal{M}(\T^3))} 
        \leq&\,C \left(\sqrt{A^{\nu}_{\epsilon}} + A^{\nu}_{\epsilon} \right).
\end{aligned} \eeq

\subsubsection{Mollification in $\epsilon$ limit}
To begin, we note that we can write 
$$ R^\nu_{\epsilon} = u^\nu \cdot \nabla \bar \theta_\epsilon^\nu - [u^\nu \cdot \nabla \bar \theta^\nu]_\epsilon.$$
We now recall a commutator lemma due to DiPerna and Lions. 
\begin{lemma}[Lemma II.1 of \cite{diperna1989ordinary}] \label{lemma:DLcommutator}
    Let $1\leq p\leq \infty$ and $1\leq \gamma< \infty$. If $u \in L^\gamma(0,T; W^{1,\alpha}(\T^3))$, for an $\alpha > p'$, and $ \psi \in L^\infty([0,T]; L^p(\T^3))$
    then 
    $$ u \cdot \nabla \psi_\epsilon- [u\cdot \nabla \psi]_\epsilon \to 0\quad \text{in }L^\gamma([0,T]; L^\beta(\T^3)) \quad \text{as }\epsilon \to 0, $$
    where $\frac{1}{\beta} = \frac{1}{\alpha} + \frac{1}{p}$ when either $\alpha < \infty$ or $p < \infty$.
\end{lemma}
Applying \Cref{lemma:DLcommutator} with $\gamma =\alpha = \beta= 2$ and $p = \infty$ we can immediately deduce that 
$$ R^\nu_{\epsilon} \to 0  \quad \text{in }L^2([0,T]; L^2(\T^3)) \quad \text{as }\epsilon \to 0. $$
Since $\rho^\nu$ and  $(-\rho s_M)^\nu$ are uniformly bounded in $L^2(0,T; L^2(\T^3))$ we can conclude that 
$$ J_2,J_3 \to 0.$$

To take the limit of~\eqref{eq:bound-after-delta}, it remains to show that the initial value term $F_\epsilon(0,\cdot)$ converges. 
This follows from the fact that $f\ast \phi_\epsilon \to f$ a.e. for $f \in L^\infty(\T^3)$ and by using the definition~\eqref{eq:def_relative_entropy} to bound the integrand. 
We then conclude by the dominated convergence theorem. 

Therefore, taking the limit as $\epsilon \to 0$ of~\eqref{eq:bound-after-a} we find
\beq \label{eq:bound-after-eps}\begin{aligned}
     \esssup_{t \in [0,T]}\int_{\T^3} \mathcal{G}(m^{\nu}) + \mathcal{F}(\rho^{\nu}) \,dx &+ \norm{\mathcal{E}^{\nu} -  \bar{\mathcal{E}} }_{L^\infty(0,T; \mathcal{M}(\T^3))} \\
        \leq&\,C \left(\sqrt{\int_{\T^3} (-\rho s_M)(U_0| V) \,dx} + \int_{\T^3} (-\rho s_M)(U_0| V) \,dx \right).
\end{aligned} \eeq

\subsubsection{Viscosity limit {($\nu \to 0)$}}
We begin by recalling that $\norm{\rho^\nu}_{L^\infty(0,T; L^\gamma(\T^3))} < C$ by~\eqref{eq:bounds-from-energy} where $C$ is independent of $\nu$. 
As a result, there exists a function $\rho$ such that 
\beq \label{eq:nu-rho-lim} \rho^\nu \wto \rho \quad\text{weakly-$\ast$ in}\quad L^\infty(0,T; L^\gamma(\T^3))\quad\text{as}\quad \nu \to 0. \eeq
Next, we recall that by~\eqref{eq:m-uniform-bound} we have a uniform bound of $m^\nu$ in $L^\infty(0,T; L^{2\gamma/(\gamma + 1)})$.
As a result, there exists $m$ such that 
$$ m^\nu \wto m \quad\text{weakly-$\ast$ in}\quad L^\infty(0,T; L^{2\gamma/(\gamma + 1)}(\T^3))\quad\text{as}\quad \nu\to 0. $$
Finally, we note~\eqref{eq:bound-after-eps} gives a uniform bound $ \norm{\mathcal{E}^\nu - \bar{\mathcal{E}}}_{L^\infty(0,T; \mathcal{M}(\T^3))},$ hence we have a $\mathcal{E}$ such that 
$$ \mathcal{E}^\nu \wto \mathcal{E} \quad\text{weakly-$\ast$ in}\quad L^\infty(0,T; \mathcal{M}(\T^3)) \quad \text{as}\quad \nu\to 0.$$
Therefore, taking the limit of~\eqref{eq:bound-after-eps} as $\nu\to 0$ we find
$$\begin{aligned}
     \esssup_{t \in [0,T]}\int_{\T^3} \mathcal{G}(m) + \mathcal{F}(\rho) \,dx &+\norm{\mathcal{E}- \bar{\mathcal{E}}}_{L^\infty(0,T; \mathcal{M}(\T^3))} \\
        \leq&\,C \left(\sqrt{\int_{\T^3} (-\rho s_M)(U_0| V) \,dx} + \int_{\T^3} (-\rho s_M)(U_0| V) \,dx \right),
\end{aligned} $$
for some $C > 0$, which is exactly the expression~\eqref{eq:main-est}.

\subsection{Uniqueness of vanishing dissipation limits} Consider a sequence of initial data $\{U^n_0\}_{n=1}^\infty$ such that 
\begin{equation} \label{eq:rel-ent-cv}
    \lim_{n\to \infty}\int_{\T^3} (-\rho s_M)(U^n_0| V_0)\,dx = 0,
\end{equation}
where $V$ is the entropy wave~\eqref{eq:V-form}. 
We note that this convergence is implied by~\eqref{eq:uniform-bound-assumption-thm}.
Then by the estimate~\eqref{eq:main-est} the vanishing dissipation limits $(\rho^n, m^n, \mathcal{E}^n)$ with initial data $U^n_0$ satisfy 
\begin{equation}\label{eq:lim-of-co-quants} 
    \lim_{n\to \infty} \left[ \esssup_{t \in [0,T]} \int_{\T^3} \mathcal{G}(m^n) + \mathcal{F}(\rho^n) \,dx  +\norm{\mathcal{E}^{n} -  \bar{\mathcal{E}} }_{L^\infty(0,T; \mathcal{M}(\T^3))} \right]  = 0. 
\end{equation}
We note that by~\eqref{eq:lim-of-co-quants} we have convergence of total energy,  
$$ \mathcal{E}^n \to \bar {\mathcal{E}}\quad \text{in}\quad L^\infty(0,T;\mathcal{M}(\T^3)) \quad \text{as}\quad n \to \infty.$$
Furthermore, from equation~\eqref{eq:rel-ent-cv} and \Cref{lemma:coercive} we have the uniform bound 
$$ \norm{\mathcal{E}^n_0}_{L^1(\T^3)} < C.$$
Since $\rho^n$ and $m^n$ are obtained as weak limits from bounds depending on the $\norm{\mathcal{E}_0^n}_{L^1(\T^3)}$, we similarly have
\begin{equation} \label{eq:limits-uni-bounds-energy}
    \norm{\rho^n}_{L^\infty(0,T; L^\gamma(\T^3))} + \norm{m^n}_{L^\infty(0,T; L^{2\gamma/(\gamma + 1)}(\T^3))} < C.
\end{equation}
In particular, up to a subsequence we have a $\rho$ such that 
$$\rho^n \wto \rho \quad \text{weakly$-\ast$ in} \quad L^\infty(0,T; L^\gamma(\T^3))\quad \text{as}\quad n\to \infty.$$ 
For $m$, we note that from the bound~\eqref{eq:L1-G-control} we have
$$ \norm{m^n}_{L^\infty(0,T;L^1(\T^3))} \leq C \esssup_{t\in [0,T]} \left(\int_{\T^3} \mathcal{G}(m)\,dx + \sqrt{\int_{\T^3} \mathcal{G}(m)\,dx} \right) \to 0\quad\text{as}\quad n\to\infty.$$
Interpolating this limit with the momentum bound~\eqref{eq:limits-uni-bounds-energy} we see 
$$m \to 0 \quad \text{in} \quad L^\infty(0,T; L^q(\T^3)) \quad \text{for any} \quad q < \frac{2\gamma}{\gamma + 1}.$$


To deduce that the weak limit $\rho$ is in fact $\bar \rho$ we first note that since $\rho^n$ and $m^n$ are weak limits from the Navier--Stokes--Fourier system they satisfy the continuity equation,
\beq \label{eq:limit-cont}  
    \partial_t \rho^n + \div(m^n) = 0 \quad \text{in}\quad \mathcal{D}'([0,T]\times \T^3).
\eeq
To begin, we test the continuity equation~\eqref{eq:limit-cont} against a test function of the form $\Phi(x)\chi_\beta(s)$ for any $\Phi \in C^\infty(\T^3)$ and 
\begin{equation}\label{eq:chi-test-fnc}\chi_\beta(t) = 1 - \int_0^t \varphi_\beta(T - s)\,ds,\end{equation}
where $\varphi_\beta$ is a suitable 1-d mollifier, similar to what was defined in \Cref{sec:shift}.
This yields 
$$ \int_0^\infty \int_{\T^3} \Phi(x) \chi_\beta'(s) \rho^n(s,x)\,dx\,ds + \int_0^\infty \int_{\T^3}\chi_\beta(s) \nabla \Phi(x)\cdot m^n(s,x)\,dx\,ds =-\int_{T^3} \Phi(x) \rho^n(0,x)\,dx.  $$
We next note that since $\rho^n \in L^\infty(0,T; L^\gamma(\T^3))$, $\partial_t \rho^n \in L^\infty(0,T; W^{-1, 2\gamma/(\gamma + 1)}(\T^3))$ for a $\gamma > 2$ we have the sequence of embeddings $L^\gamma(\T^3) \hookrightarrow\hookrightarrow H^{-1}(\T^3) \hookrightarrow W^{-1, 2\gamma/ (\gamma + 1)}$.
By Aubin-Lions we can then conclude $\rho^n \in C([0,T]; H^{-1}(\T^3) )$ and there is a subsequence that converges strongly to some $\rho$ in this space. 
First, passing to the limit $n\to \infty$ we find 
$$ \int_0^\infty\int_{\T^3} \Phi(x) \chi_\beta'(s) \rho(t,x)\,dx\,ds = -\int_{T^3} \Phi(x) \rho(0,x)\,dx $$
due to $m^n \to 0$ in $L^\infty(0,T;L^{q}(\T^3))$ for some $q > 1$. 
We now remark that due to~\eqref{eq:rel-ent-cv} we have $\rho_0^n \to \bar \rho$ distributionally, because
\begin{align*} 
    \int_{\T^3} \Phi(x) (\rho_n^0 - \bar \rho)\,dx  &=\int_{\T^3} \Phi(x) (\rho_n^0 - \bar \rho)(\mathbbm{1}_{ \{ |\rho_n^0 - \bar \rho| < 1 \}} + \mathbbm{1}_{ \{ |\rho_n^0 - \bar \rho| \geq 1 \}})\,dx \\
    &\leq \norm{\Phi}_{L^\infty} \int_{\{ |\rho_n^0 - \bar \rho| \geq 1 \}} |\rho_n^0 - \bar \rho|\,dx + \norm{\Phi}_{L^2}\left(\int_{\{ |\rho_n^0 - \bar \rho| < 1 \}} |\rho_n^0 - \bar \rho|^2\,dx\right)^{1/2} \\
    &\lesssim_{\Phi} \int_{\T^3} \mathcal{G}(\rho_0^n - \bar \rho) \,dx + \sqrt{\int_{\T^3} \mathcal{G}(\rho_0^n - \bar \rho) \,dx} \\
    &\leq C(-\rho s)(U^n_0| \bar U) + C\sqrt{(-\rho s)(U^n_0| \bar U)}  \to 0.
\end{align*}
Now using $\rho^n_0 \to \bar \rho$ and the time continuity, we pass to the limit in $n$ and $\beta$ to reveal that
$$\int_{\T^3} \Phi(x) \rho(t,x)\,dx = \int_{\T^3} \Phi(x) \bar \rho(x)\,dx. $$
We conclude $ \rho(t,\cdot) = \bar \rho$ a.e. in $\T^3$ for all $t \in [0,T]$. 

\begin{remark} \label{rmk:strong-cv-density}
    Suppose that our entropy wave takes just two states,
    $$ V(t,x) = (\rho^*,0,\theta^*)\mathbbm{1}_{\Omega} + (\rho_*,0,\theta_*)\mathbbm{1}_{\T^3\setminus\Omega}$$
    on some measurable set $\Omega$. 
    In this case, we can establish $\rho^n\to \rho$ in $L^\infty(0,T;L^1( \T^3))$. 
    In light of our uniform bounds on $\rho^n$, this allows us to conclude $\rho^n \to \rho$ in $L^\infty(0,T; L^p(\T^3))$ for all $1\leq p < \gamma$ by interpolation.

Indeed, restricting to $\Omega$ where $\bar \rho = \rho^\ast$ and using that $\rho^n,\rho \in C(0,T; H^{-1}(\T^3))$ we find 
\begin{align*} 
    \int_\Omega |\rho^n(t,x) - \rho^\ast| \,dx &= \int_\Omega (\rho^n(t,x) - \rho^\ast)_+ + (\rho^n(t,x) - \rho^\ast)_- \,dx \\
    &= \int_\Omega  ((\rho^n(t,x) - \rho^\ast)_- - (\rho^n(t,x) - \rho^\ast)_+) + 2(\rho^n(t,x) - \rho^\ast)_+ \,dx \\
    &= -\int_\Omega(\rho^n(t,x) - \rho^\ast) \,dx + \int_\Omega 2(\rho^n(t,x) - \rho^\ast)_+\,dx  \to 0
\end{align*}
where the first integral converges to zero via the weak convergence $\rho^n \wto \rho$ weak-$\ast$ in $L^\infty(0,T; L^\gamma(\T^3))$ while the second converges to zero due to the control
\begin{align*}
     \int_\Omega (\rho^n (t,x)- \rho^\ast)_+\,dx &= \int_\Omega (\rho^n(t,x) - \rho^\ast)_+ (\mathbbm{1}_{\{|\rho^n(t,x) -\rho^\ast| > 1\}}(x) + \mathbbm{1}_{\{|\rho^n(t,x) -\rho^\ast| \leq 1\}}(x))\,dx \\
     &\leq \int_\Omega (\rho^n(t,x) - \rho^\ast)_+ \mathbbm{1}_{|\rho^n(t,x) -\rho^\ast| > 1}(x)\,dx\\
     &\qquad +  \left(\int_\Omega (\rho^n(t,x) - \rho^\ast)_+^2\mathbbm{1}_{|\rho^n(t,x) -\rho^\ast| \leq 1}\,dx \right)^{1/2} \\
     &\leq \int_{\T^3} \mathcal{F}(\rho^n) \,dx + \left(\int_{\T^3} \mathcal{F}(\rho^n) \,dx\right)^{1/2} \to 0
\end{align*}
where the last limit is due to equation~\eqref{eq:lim-of-co-quants}.
We can do the same within $\Omega^c$ showing the desired convergence $\rho^n \to \bar\rho$ in $L^\infty(0,T; L^1(\T^3))$.
\end{remark}

\section{AI Use Disclosure} Generative AI was used to assist with the final editing and proofreading of this manuscript. All results, arguments, and proofs are made by the authors.

\printbibliography

@article{kvy21,
  title={Uniqueness of a planar contact discontinuity for 3D compressible Euler system in a class of zero dissipation limits from Navier--Stokes--Fourier system},
  author={Kang, Moon-Jin and Vasseur, Alexis F and Wang, Yi},
  journal={Communications in Mathematical Physics},
  volume={384},
  number={3},
  pages={1751--1782},
  year={2021},
  publisher={Springer}
}

@article{novotnyexist13,
  title={Navier--Stokes--Fourier system on unbounded domains: Weak solutions, relative entropies, weak-strong uniqueness},
  author={Jessl{\'e}, Didier and Jin, Bum Ja and Novotny, Antonin},
  journal={SIAM Journal on Mathematical Analysis},
  volume={45},
  number={3},
  pages={1907--1951},
  year={2013},
  publisher={SIAM}
}

@book{feireisl04,
  title={Dynamics of viscous compressible fluids},
  author={Feireisl, Eduard},
  volume={26},
  year={2004},
  publisher={Oxford University Press}
}

@book{feireislnovotny17,
  title={Singular Limits in Thermodynamics of Viscous Fluids},
  author={Feireisl, Eduard and Novotny, Antonin}, 
  year={2017},
  publisher={Birkh{\"a}user}
}

@misc{maria-convex,
      title={A note on relations between convexity and concavity of thermodynamic functions}, 
      author={Mária Lukáčová-Medvid'ová and Ferdinand Thein and Gerald Warnecke and Yuhuan Yuan},
      year={2025},
      eprint={2510.24440},
      archivePrefix={arXiv},
      primaryClass={math.AP},
      url={https://arxiv.org/abs/2510.24440}, 
}

@article {feireisl-zero-diss-limit,
    AUTHOR = {Feireisl, Eduard},
     TITLE = {Vanishing dissipation limit for the
              {N}avier-{S}tokes-{F}ourier system},
   JOURNAL = {Commun. Math. Sci.},
  FJOURNAL = {Communications in Mathematical Sciences},
    VOLUME = {14},
      YEAR = {2016},
    NUMBER = {6},
     PAGES = {1535--1551},
      ISSN = {1539-6746,1945-0796},
   MRCLASS = {35Q35 (35D30 35Q31)},
  MRNUMBER = {3538361},
       DOI = {10.4310/CMS.2016.v14.n6.a4},
       URL = {https://doi.org/10.4310/CMS.2016.v14.n6.a4},
}

@article{wang2024nonlinear,
  title={Nonlinear stability of entropy waves for the Euler equations},
  author={Wang, Wei and Zhang, Zhifei and Zhao, Wenbin},
  journal={Mathematische Annalen},
  volume={390},
  number={4},
  pages={4861--4888},
  year={2024},
  publisher={Springer}
}

@article{krupa2025contact,
  title={Contact discontinuities for 2-D isentropic Euler are unique in 1-D but wildly non-unique otherwise},
  author={Krupa, Sam G and Sz{\'e}kelyhidi Jr, L{\'a}szl{\'o}},
  journal={Communications in Mathematical Physics},
  volume={406},
  number={5},
  pages={109},
  year={2025},
  publisher={Springer}
}

@article{de2010admissibility,
  title={On admissibility criteria for weak solutions of the Euler equations},
  author={De Lellis, Camillo and Sz{\'e}kelyhidi Jr, L{\'a}szl{\'o}},
  journal={Archive for rational mechanics and analysis},
  volume={195},
  number={1},
  pages={225--260},
  year={2010},
  publisher={Springer}
}

@article{chiodaroli2015global,
  title={Global ill-posedness of the isentropic system of gas dynamics},
  author={Chiodaroli, Elisabetta and De Lellis, Camillo and Kreml, Ond{\v{r}}ej},
  journal={Communications on Pure and Applied Mathematics},
  volume={68},
  number={7},
  pages={1157--1190},
  year={2015},
  publisher={Wiley Online Library}
}

@article{mengual2023dissipative,
  title={Dissipative Euler flows for vortex sheet initial data without distinguished sign},
  author={Mengual, Francisco and Sz{\'e}kelyhidi Jr, L{\'a}szl{\'o}},
  journal={Communications on Pure and Applied Mathematics},
  volume={76},
  number={1},
  pages={163--221},
  year={2023},
  publisher={Wiley Online Library}
}

@article{vy-1d,
  title={The inviscid limit to a contact discontinuity for the compressible Navier--Stokes--Fourier system using the relative entropy method},
  author={Vasseur, Alexis and Wang, Yi},
  journal={SIAM Journal on Mathematical Analysis},
  volume={47},
  number={6},
  pages={4350--4359},
  year={2015},
  publisher={SIAM}
}

@article {MR546634,
    AUTHOR = {Dafermos, Constantine M.},
     TITLE = {The second law of thermodynamics and stability},
   JOURNAL = {Arch. Rational Mech. Anal.},
  FJOURNAL = {Archive for Rational Mechanics and Analysis},
    VOLUME = {70},
      YEAR = {1979},
    NUMBER = {2},
     PAGES = {167--179},
      ISSN = {0003-9527},
   MRCLASS = {73B30 (35L65)},
  MRNUMBER = {546634},
MRREVIEWER = {C.-C.\ Wang},
       DOI = {10.1007/BF00250353},
       URL = {https://doi.org/10.1007/BF00250353},
}

@article {MR523630,
    AUTHOR = {DiPerna, Ronald J.},
     TITLE = {Uniqueness of solutions to hyperbolic conservation laws},
   JOURNAL = {Indiana Univ. Math. J.},
  FJOURNAL = {Indiana University Mathematics Journal},
    VOLUME = {28},
      YEAR = {1979},
    NUMBER = {1},
     PAGES = {137--188},
      ISSN = {0022-2518,1943-5258},
   MRCLASS = {35L65},
  MRNUMBER = {523630},
MRREVIEWER = {V.\ \v Durikovi\v c},
       DOI = {10.1512/iumj.1979.28.28011},
       URL = {https://doi.org/10.1512/iumj.1979.28.28011},
}

@article{serre2014relative,
  title={About the relative entropy method for hyperbolic systems of conservation laws},
  author={Serre, Denis and Vasseur, Alexis F},
  booktitle={Contemp. Math. AMS},
  volume={658},
  pages={237--248},
  year={2016},
  organization={American Mathematical Society}
}

@article{serre20142,
  title={$ {L}^2$-type contraction for systems of conservation laws},
  author={Serre, Denis and Vasseur, Alexis F},
  journal={Journal de l’{\'E}cole polytechnique—Math{\'e}matiques},
  volume={1},
  pages={1--28},
  year={2014}
}

@article {MR2807139,
    AUTHOR = {Leger, Nicholas and Vasseur, Alexis},
     TITLE = {Relative entropy and the stability of shocks and contact
              discontinuities for systems of conservation laws with non-{BV}
              perturbations},
   JOURNAL = {Arch. Ration. Mech. Anal.},
  FJOURNAL = {Archive for Rational Mechanics and Analysis},
    VOLUME = {201},
      YEAR = {2011},
    NUMBER = {1},
     PAGES = {271--302},
      ISSN = {0003-9527,1432-0673},
   MRCLASS = {35L65 (35A02 35B35 35D30 35L67 35Q31 76L05)},
  MRNUMBER = {2807139},
MRREVIEWER = {Ramon\ G.\ Plaza},
       DOI = {10.1007/s00205-011-0431-1},
       URL = {https://doi.org/10.1007/s00205-011-0431-1},
}

@article{kang2016criteria,
  title={Criteria on contractions for entropic discontinuities of systems of conservation laws},
  author={Kang, Moon-Jin and Vasseur, Alexis F},
  journal={Archive for Rational Mechanics and Analysis},
  volume={222},
  number={1},
  pages={343--391},
  year={2016},
  publisher={Springer}
}

@article{NSVS,
    DOI={10.48550/arXiv.2401.09305},
    url={https://doi.org/10.48550/arXiv.2401.09305},
    year=2024,
    month={1},
    AUTHOR={Chen, Geng and Kang, Moon-Jin and Vasseur, Alexis F.},
    title={From {N}avier-{S}tokes to {BV} solutions of the barotropic Euler equations},
    journal = {Preprint on Arxiv}
}

@article{kang2020contraction,
  title={Contraction property for large perturbations of shocks of the barotropic Navier--Stokes system},
  author={Kang, Moon-Jin and Vasseur, Alexis F},
  journal={Journal of the European Mathematical Society},
  volume={23},
  number={2},
  pages={585--638},
  year={2021}
}

@article{kang2025time,
  title={Time-asymptotic stability of generic Riemann solutions for compressible Navier--Stokes--Fourier equations},
  author={Kang, Moon-Jin and Vasseur, Alexis F and Wang, Yi},
  journal={Archive for Rational Mechanics and Analysis},
  volume={249},
  number={4},
  pages={42},
  year={2025},
  publisher={Springer}
}

@article{giri2022non,
  title={On Non-uniqueness of Continuous Entropy Solutions to the Isentropic Compressible Euler Equations},
  author={Giri, Vikram and Kwon, Hyunju},
  journal={Archive for Rational Mechanics and Analysis},
  volume={245},
  number={2},
  pages={1213--1283},
  year={2022},
  publisher={Springer}
}

@article{chiodaroli2025long,
  title={On the long-time behavior of solutions to the Navier--Stokes--Fourier system on unbounded domains},
  author={Chiodaroli, Elisabetta and Feireisl, Eduard},
  journal={Journal of the London Mathematical Society},
  volume={111},
  number={1},
  pages={e70067},
  year={2025},
  publisher={Wiley Online Library}
}

@article{kang2021uniqueness,
  title={Uniqueness and stability of entropy shocks to the isentropic Euler system in a class of inviscid limits from a large family of Navier--Stokes systems},
  author={Kang, Moon-Jin and Vasseur, Alexis F},
  journal={Inventiones mathematicae},
  volume={224},
  number={1},
  pages={55--146},
  year={2021},
  publisher={Springer}
}

@book{markfelder2021convex,
  title={Convex integration applied to the multi-dimensional compressible Euler equations},
  author={Markfelder, Simon},
  year={2021},
  publisher={Springer}
}

@article{diperna1989ordinary,
  title={Ordinary differential equations, transport theory and Sobolev spaces},
  author={DiPerna, Ronald J and Lions, Pierre-Louis},
  journal={Inventiones mathematicae},
  volume={98},
  number={3},
  pages={511--547},
  year={1989},
  publisher={Springer}
}

@article{VASSEUR2008323,
title = "Recent Results on Hydrodynamic Limits",
journal = "Handbook of Differential Equations: Evolutionary Equations",
volume = "4",
number = "",
pages = "323 - 376",
year = "2008",
note = "",
issn = "1874-5717",
doi = "http://dx.doi.org/10.1016/S1874-5717(08)00007-8",
url = "http://www.sciencedirect.com/science/article/pii/S1874571708000078",
author = "Vasseur, Alexis F."
}
\end{document}